\documentclass{scrartcl}
\usepackage{graphicx} 
\usepackage{amsmath}
\usepackage{amsthm}
\usepackage{amsfonts}
\usepackage{amssymb}
\usepackage{array}
\usepackage{nicematrix}
\usepackage{subfiles}
\usepackage{lmodern}
\usepackage{comment}
\usepackage{listings}
\usepackage{mathtools}
\usepackage{hyperref}
\usepackage{changepage}
\usepackage{faktor}
\usepackage{caption}
\usepackage{float}
\usepackage{tipa}
\usepackage[OT1]{fontenc}

\usepackage{enumitem}
\usepackage[a4paper,total={17cm, 26cm}]{geometry}
\usepackage[backend=biber,style=alphabetic, maxnames=500, maxalphanames=500, doi=false, isbn=false, url=false, eprint=false]{biblatex}
\usepackage{nicematrix}
\renewbibmacro*{in:}{}

\newtheorem{Definition}{Definition}[section]
\newtheorem{Remark}[Definition]{Remark}
\newtheorem{Example}[Definition]{Example}
\newtheorem{Theorem}[Definition]{Theorem}
\newtheorem{Proposition}[Definition]{Proposition}
\newtheorem{Lemma}[Definition]{Lemma}
\newtheorem{Corollary}[Definition]{Corollary}

\mathtoolsset{centercolon}
\title{Wild frieze patterns over the integers}
\date{}
\author{Leon Eickhoff \footnote{\hspace{0.18em} Institut für Algebra, Zahlentheorie und Diskrete Mathematik, Institut für Algebraische Geometrie\\ Leibniz Universität
Hannover, Welfengarten 1, 30167 Hannover, Germany
\\
\noindent
E-Mail: eickhoff@math.uni-hannover.de
\\
}
}
\begin{document}

\maketitle
\begin{abstract}
    \noindent
    Abstract. \textit{We study frieze patterns over the integers that are allowed to have wild entries. We introduce the quiddity number as a new invariant. The quiddity number is then used to classify strongly connected components of the directed graph $\Gamma_{2,n}(\mathbb{Z})$. Furthermore, we show that every finite simple directed graph arises as an induced subgraph of a directed graph $\Gamma_{2,n}(\mathbb{Z})$ for $n$ sufficiently large.}
\end{abstract}
\section{Introduction}
Frieze patterns were first introduced in \cite{Cox71} by Coxeter. Conway and Coxeter related frieze patterns with entries in the positive integers to triangulations of convex polygons in \cite{CC73a}. 
Almost 30 years later, Fomin and Zelevinsky introduced cluster algebras \cite{FZ02}. These algebras provide a common framework for different recurrence relations satisfying the ``Laurent Phenomenon". Frieze patterns are one of many examples where this curiosity shows up \cite[Example 4.3]{FOMIN2002119}. 
The relationship between frieze patterns and cluster algebras was later made more explicit in \cite{CC06}.
These developments reignited the interest in frieze patterns.\\
Since then, frieze patterns have been generalized in several ways. \\
Most results on frieze patterns require either implicitly (by allowing only non-zero entries) or explicitly that the frieze patterns are tame:
Tame frieze patterns over a commutative ring $R$ can be parametrized by semiclosed paths in the Farey complex $\mathcal{E}_R$ \cite[Theorem 1.3]{SSZ2025}. 
The variety of tame $SL_{k}$-frieze patterns of height $n$ was shown to be isomorphic to the variety of $n+k+1$ superperiodic linear difference equations of order $k$ \cite[Theorem 3.4.1]{morier2014linear}. It is also shown there that if $n+1$ and $k$ are coprime, these varieties are isomorphic to the moduli space of non-degenerate $n+k+1$-gons in $\mathbb{P}^{k-1}$
up to projective equivalence. All tame frieze patterns over the integers can be obtained from admissible labelings \cite[Theorem 7.5]{CH19}.

Wild frieze patterns have received comparably little attention. In \cite{Cun17} some examples of wild $SL_3$-frieze patterns were given. These examples show that wild frieze patterns in general will not have some of the nice properties of tame frieze patterns.
\\
Recently, Zabolotskii showed that the maximal possible density of wild entries in an $SL_2$-tiling over an integral domain is $\frac{2}{5}$ \cite{Zab25}.\\
\\
In this article we want to consider frieze patterns over the integers that are allowed to have wild entries.
First, we gather some useful facts about (tame) frieze patterns. 
In chapter \ref{section_wild} we observe that a wild frieze pattern can be obtained by
gluing together tame frieze patterns with common rows. This leads to the definition of the quiddity number of a frieze pattern.
For tame frieze patterns, this is just the sum of the entries of the quiddity cycle.
\newpage
Using this new invariant, we are able to classify the strongly connected components (which are the same as the weakly connected components) of the directed graph $\Gamma_{2,n}(\mathbb{Z})$ (See Definition \ref{definition Graph}):

\begin{Theorem}(Theorem \ref{connected component classification})
       Let $v$ and $w$ be rows of frieze patterns of height $n$. Then $v$ and $w$ are in the same strongly connected component of $\Gamma_{2,n}(\mathbb{Z})$ if and only if one of the two following conditions holds:
    \begin{enumerate}
        \item $Q(v)=Q(w)=\pm (3n+3)$, and $v$ and $w$ are both rows of the same (twisted) Conway-Coxeter frieze pattern.
        \item $Q(v)=Q(w)\neq \pm(3n+3)$.
    \end{enumerate}
    
\end{Theorem}

Using this classification, we are able to enumerate the strongly connected components in Proposition \ref{counting connected components}. 
Since (possibly wild) frieze patterns over $\mathbb{Z}$ correspond to infinite paths in $\Gamma_{2,n}(\mathbb{Z})$, this Theorem can be used to decide whether two tuples of integer can appear as rows of the same frieze pattern.
This graph (up to some minor modifications) was introduced in \cite{Cun17} in order to find examples of wild frieze patterns with the help of a computer. \\
In chapter \ref{section_properties} we further investigate $\Gamma_{2,n}(\mathbb{Z})$. We characterize loops and directed $2$-cycles in $\Gamma_{2,n}(\mathbb{Z})$ and show that every finite simple directed graph arises as an induced subgraph of $\Gamma_{2,n}(\mathbb{Z})$ for a sufficiently large $n$.
\section*{Acknowledgement}

I want to thank my advisor Michael Cuntz for his suggestion to look at wild frieze patterns and for many helpful discussions. I want to thank Thorsten Holm for introducing me to the topic of frieze patterns during my master's thesis. 
Furthermore, I want to thank Andrei Zabolotskii for pointing me to his slides about Farey surfaces of (wild) $SL_2$-tilings, and Sean A. Irvine for noticing a superfluous term in an earlier version of Proposition \ref{counting connected components}. \\
The author received funding through the RTG 2965 - From Geometry to Numbers: Moduli, Hodge Theory, Rational Points, Project number 512730679.\\
\\
The results were obtained and the article was written by the (human) author and not by an AI tool. However, a chatbot was used to find errors present in an earlier version.
\section{Generalities on frieze patterns}
\begin{Definition} \label{Def Frieze Pattern}
    Let $R$ be a commutative ring, $S\subseteq R$ and $n\in \mathbb{Z}_{\geq 0}$. A frieze pattern $\mathcal{F}$ of height $n$ over $S$ consists of entries $c_{i,j} \in S\cup\{0,1\}$ for $i,j \in \mathbb{Z}$, $i\leq  j\leq i+n+3$ such that \begin{enumerate}
        \item $c_{i,i}=c_{i,i+n+3}=0$ for all $i \in \mathbb{Z}$.
        \item $c_{i,i+1}=c_{i,i+n+2}=1$ for all $i \in \mathbb{Z}$.
        \item $c_{i,j}\in S$ for $i\in \mathbb{Z}$ and $i+2\leq j \leq i+n+1$.
        \item $c_{i,j}c_{i+1,j+1}-c_{i,j+1}c_{i+1,j}=1$ for all $i,j \in \mathbb{Z}$ with $i+1\leq j \leq i+n+2$. \label{SL_2}
    \end{enumerate}
    \phantom{""}
    \\
    \\
    They can be depicted in the following way:
    \begin{align*}
    \begin{NiceMatrixBlock}[auto-columns-width]
        \begin{NiceMatrix}
             \ddots &&&&\ddots \\\\
            0&1&c_{i-1,i+1}&\dots &c_{i-1,i+n}&1&0\\\\
            &0&1&c_{i,i+2}&\dots&c_{i,i+n+1}&1&0\\\\
            &&0&1&c_{i+1,i+3} &\dots&c_{i+1,i+n+2} &1&0\\\\
            &&&&\ddots&&&&\ddots
        \end{NiceMatrix}
        \end{NiceMatrixBlock}
    \end{align*}
    We will refer to condition \ref{SL_2}. as the $SL_2$ rule.
    \\
    The $i$-th row of a frieze pattern consists of all the entries $c_{i,j}$ with $i\leq j \leq i+ n+3$. Conversely, the $j$-th column of a frieze pattern consists of all entries $c_{i,j}$ with $j-n-3\leq i\leq j$. The $k$-th diagonal consists of the entries $c_{i,i+k}$ for $i\in \mathbb{Z}$.\\\\
    An entry $c_{i,j}$ with $i+2\leq j \leq i+n+1$ is called tame if 
    \begin{align*}
        \begin{vmatrix}
         c_{i-1,j-1}&c_{i-1,j}&c_{i-1,j+1}  \\
         c_{i,j-1}&c_{i,j}&c_{i,j+1} \\
         c_{i+1,j-1}&c_{i+1,j}&c_{i+1,j+1}
    \end{vmatrix}=0.
    \end{align*}
    A frieze pattern $\mathcal{F}$ is called tame if every entry $c_{i,j}$ with $i+2\leq j \leq i+n+1$ of $\mathcal{F}$ is tame. Otherwise, $\mathcal{F}$ is called wild.
\end{Definition}

\begin{Example}

\begin{itemize}
    \item Frieze patterns over $\mathbb{Z}_{>0}$ were classified by Conway and Coxeter in \cite{CC73a} and \cite{CC73b}. We will refer to them as Conway-Coxeter frieze patterns.
    \item  In \cite{Fon14} it was shown that each frieze pattern over $\mathbb{Z}\setminus\{0\}$ is either a Conway-Coxeter frieze pattern or is obtained from a Conway-Coxeter frieze pattern by multiplying every even diagonal by $-1$. Frieze patterns of the  second type only exist if $n$ is odd. We refer to them as twisted Conway-Coxeter frieze patterns. 
    \item We obtain a wild frieze pattern over $\mathbb{Z}$ by repeating the segment
    \begin{align}\label{Wild Example}
     \begin{NiceMatrixBlock}[auto-columns-width]
        \begin{NiceMatrix}
            \ddots&&&&&\ddots\\
            0&1&-3&-1&-3&-2&1&0\\
            &0&1&0&-1&-1&0&1&0\\
            &&0&1&6&5&-1&3&1&0\\
            &&&0&1&1&0&-1&0&1&0\\
            &&&&0&1&1&-3&-1&6&1&0\\
            &&&&&0&1&-2&-1&5&1&1&0\\
            &&&&&&0&1&0&-1&0&1&1&0\\
            &&&&&&&& \ddots&&&&&\ddots\\
        \end{NiceMatrix}
    \end{NiceMatrixBlock}
    \end{align}
    periodically. 
\end{itemize}
    
\end{Example}

\begin{Lemma} \label{tame entries}
    Let $\mathcal{F}$ be a frieze pattern over an integral domain $R$ with entries $c_{i,j}$. Each entry $c_{i,j}\neq 0$ is tame. If $c_{i,j}=0$, then it is tame if and only if $c_{i-1,j-1}+c_{i-1,j+1}+c_{i+1,j-1}+c_{i+1,j+1}=0.$

    \begin{proof}
        The first claim is a special case of \cite[Proposition 1]{BR10}.

        For the second claim (which also appears similarly in \cite{Zab25}) we compute
        \begin{align*}   
        &\phantom{\overset{x_{i,j}=0}{=}}
            \begin{vmatrix}
         c_{i-1,j-1}&c_{i-1,j}&c_{i-1,j+1}  \\
         c_{i,j-1}&c_{i,j}&c_{i,j+1} \\
         c_{i+1,j-1}&c_{i+1,j}&c_{i+1,j+1}
    \end{vmatrix}\\
         &\overset{\phantom{x_{i,j}=0}}{=}c_{i-1,j-1}\cdot
    \begin{vmatrix}
        c_{i,j}&c_{i,j+1}\\
        c_{i+1,j}&c_{i+1,j+1}
    \end{vmatrix}-c_{i,j-1}\cdot \begin{vmatrix}
        c_{i-1,j}&c_{i-1,j+1}  \\
        c_{i+1,j}&c_{i+1,j+1}
    \end{vmatrix}
    +c_{i+1,j-1}\begin{vmatrix}
        c_{i-1,j}&c_{i-1,j+1}\\
        c_{i,j}&c_{i,j+1} 
    \end{vmatrix}\\
    &\overset{\phantom{c_{i,j}=0}}{=}c_{i-1,j-1}+c_{i+1,j-1} -c_{i,j-1}c_{i-1,j}c_{i+1,j+1}+c_{i,j-1}c_{i+1,j}c_{i-1,j+1}\\
    &\overset{c_{i,j}=0}{=}c_{i-1,j-1}+c_{i+1,j-1} +(c_{i,j}c_{i+1,j+1}-c_{i,j-1}c_{i-1,j})c_{i+1,j+1}+(c_{i,j-1}c_{i+1,j}-c_{i-1,j-1}c_{i,j})c_{i-1,j+1}\\
     &\overset{\phantom{c_{i,j}=0}}{=}c_{i-1,j-1}+c_{i+1,j-1}+c_{i-1,j+1}+c_{i+1,j+1}.
        \end{align*}
    We used the $SL_2$ rule for the second and fourth equality.
    \end{proof}
\end{Lemma}

\begin{Definition} (\cite[(4-3)]{CH09} and \cite[Definition 6.1]{CH19})
    Let $R$ be a commutative ring and $c \in R$. We write $\eta(c):= 
    \begin{pmatrix}
        c&-1\\
        1 &0
    \end{pmatrix}$ .
    Let $\varepsilon \in \{\pm1\}$.
    A tuple $(c_1,\dots,c_m)\in R^{m}$ is called an $\varepsilon$-quiddity cycle (over $R$) if 
    \begin{equation*}
        \prod_{i=1}^{m}\eta(c_i)=\begin{pmatrix}
        \varepsilon&0\\
        0&\varepsilon
    \end{pmatrix}.
    \end{equation*} 
    We usually refer to $-1$-quiddity cycles simply as quiddity cycles.
\end{Definition}

\begin{Remark}\label{Frieze quiddity correspondence}
    The entries of a tame frieze pattern of height $n$ satisfy a glide symmetry: \\$c_{i,j}=c_{j,i+n+3}$ \cite[Corollary 22]{BR10}. In particular, this tells us that the entries are $n+3$-periodic.\\
    By \cite[Proposition 2.4]{CH19} tame frieze patterns of height $n$ correspond to quiddity cycles of length $n+3$.
    If a tame frieze pattern $\mathcal{T}$ has entries $c_{i,j}$, then $(c_{1,3},\dots,c_{n+3,n+5})$ is a quiddity cycle. We will refer to this as the quiddity cycle of $\mathcal{T}$. Moreover, the other entries of the tame frieze pattern can be recovered from the quiddity cycle through the equation
    \begin{equation}
        c_{i,j+2}=\left( \prod_{k=i}^{j}\eta(c_{k,k+2})
        \right)_{1,1}.
    \end{equation}
    One can modify the proof of \cite[Proposition 2.4]{CH19} to show that $1$-quiddity cycles correspond to infinite arrays of the form 
    \begin{align}
    \begin{NiceMatrixBlock}[auto-columns-width]\label{-1 frieze pattern}
        \begin{NiceMatrix}
             \ddots &&&&&\ddots \\\\
            0&1&c_{i-1,i+1}&\dots &c_{i-1,i+n}&-1&0\\\\
            &0&1&c_{i,i+2}&\dots&c_{i,i+n+1}&-1&0\\\\
            &&0&1&c_{i+1,i+3} &\dots&c_{i+1,i+n+2} &-1&0\\\\
            &&&&\ddots&&&&\ddots
        \end{NiceMatrix}
        \end{NiceMatrixBlock}
    \end{align}
    satisfying the $SL_2$ rule.\\
    \\
    There is a group action of the dihedral group $D_{n+3}$ on the set of tame frieze patterns of height $n$ because quiddity cycles can be both rotated and reversed. (See for example \cite[Remark 2.6]{CH19} for more details.)
\end{Remark}

\begin{Lemma}\label{quiddity from rows}
    Let $\mathcal{T}$ be the tame frieze pattern over an integral domain $R$ containing the two rows
    \begin{align*}
        \begin{NiceMatrixBlock}[auto-columns-width]
            \begin{NiceMatrix}
            0&1&c_{i,i+2}&\dots&c_{i,i+n+1}&1&0  \\
            &0&1&c_{i+1,i+3}&\dots&c_{i+1,i+n+2}&1&0
        \end{NiceMatrix}   .
    \end{NiceMatrixBlock}
    \end{align*}
    Then the entries of the quiddity cycle $(c_{1,3},\dots,c_{n+3,n+5})$ of $\mathcal{T}$ can be computed from these rows by 
    \begin{align*}
        c_{j,j+2}=
        \begin{cases}
            \frac{c_{i,j}+c_{i,j+2}}{c_{i,j+1}}&\text{ if } c_{i,j+1}\neq 0\\
            \frac{c_{i+1,j}+c_{i+1,j+2}}{c_{i+1,j+1}}&\text{ if }c_{i+1,j+1}\neq 0
        \end{cases}
    \end{align*}
    and the periodicity $c_{i,i+2}=c_{i+n+3,i+n+5}$.
    \begin{proof}
        The entries of a tame frieze pattern satisfy the recursion 
        \begin{equation}\label{recursion}
        c_{i,j+2}=c_{j,j+2}c_{i,j+1}-c_{i,j}.   
        \end{equation}
        (See for example the proof of \cite[Proposition 2.4]{CH19}.)
        At most one of $c_{i,j+1}$ and $c_{i+1,j+1}$ can be $0$ because otherwise the entries would violate the $SL_2$ rule.
        If $c_{i,j+1}\neq0$, we can rearrange \eqref{recursion}  to obtain $c_{j,j+2}=\frac{c_{i,j}+c_{i,j+2}}{c_{i,j+1}}$. Otherwise, we use the recursion with $i+1$ instead of $i$.\\
        This way, we obtain the values for the entries $c_{j,j+2},\dots,c_{j+n+3,j+n+5}$. The periodicity then gives the values for the quiddity cycle.
    \end{proof}
\end{Lemma}

\begin{Lemma} \label{two rows determine tame frieze}
    Let the segment
    \begin{align}\label{two_rows}
    \begin{NiceMatrixBlock}[auto-columns-width]
    \begin{NiceMatrix}
         0&1&c_{i,i+2}&\dots&c_{i,i+n+1}&1&0  \\
         &0&1&c_{i+1,i+3}&\dots&c_{i+1,i+n+2}&1&0
    \end{NiceMatrix}
     \end{NiceMatrixBlock}
    \end{align}
    with entries in $\mathbb{Z}$  satisfy the $SL_2$ condition everywhere. Then there is a unique tame frieze pattern $\mathcal{T}$ over $\mathbb{Z}$ that contains
    \eqref{two_rows}
     as the $i$-th and $i+1$st row.
\begin{proof}
    We first show that there is such a tame frieze pattern over $\mathbb{Q}$.
    We can compute the entries $c_{i+2,i+4},\dots,c_{i+2,i+n+3}$ of the third row from left to right as follows. \\If $c_{i+1,j-1} \neq 0$, then $c_{i+2,j}=\frac{c_{i+1,j}c_{i+2,j-1}+1}{c_{i+1,j-1}}$.\\ Otherwise, by Lemma \ref{tame entries} we have $c_{i+2,j}=-(c_{i,j-2}+c_{i,j}+c_{i+2,j-2})$. We can proceed the same way for all rows below. The rows above can either be computed similarly or are obtained from the entries below by using the periodicity.
    \\
    It remains to show that all the entries actually lie in $\mathbb{Z}$. By Remark \ref{Frieze quiddity correspondence} it is enough to show that the quiddity cycle entries lie in $\mathbb{Z}$. We can write $c_{j,j+2}=\frac{p}{q}$ with $p,q \in \mathbb{Z}$ coprime. 
    The recursion \eqref{recursion} shows that 
    \begin{align*}
        qc_{i,j+2}+qc_{i,j}=pc_{i,j+1}.    
    \end{align*}
    Since $p$ and $q$ are coprime, we have $q \mid c_{i,j+1}$. Similarly, we get that $q \mid c_{i+1,j+1}$. Together this shows that $q \mid (c_{i,j}c_{i+1,j+1}-c_{i,j+1}c_{i+1,j})$. But this is $1$ by the $SL_2$ rule and it follows that $q=\pm1$ and $c_{j,j+2} \in \mathbb{Z}$.
\end{proof}
\end{Lemma}

\begin{Remark}
    The proof of Lemma \ref{two rows determine tame frieze} shows that a frieze pattern with nonzero entries is already determined by one row. This is because the upper row is only used if $c_{i+1,j-1}=0$. Alternatively, this follows from the observation that all entries are evaluations of cluster variables of a type $\mathcal{A}_n$ cluster algebra and the fact that all cluster variables can be expressed as a Laurent polynomial in the cluster variables of an initial cluster \cite[Theorem 3.1]{FZ02}.
\end{Remark}
\hspace{30pt}
\\
We now introduce admissible labelings, which we will later need to prove our main result.

In \cite{CC73a} and \cite{CC73b} Conway and Coxeter established a bijection between frieze patterns of height $n$ over $\mathbb{Z}_{\geq1}$ and triangulations of convex $(n+3)$-gons by noncrossing diagonals. This was generalized in \cite{CH19} to describe tame frieze patterns over $\mathbb{Z}$.

From now on we will simply write triangulation instead of triangulation by noncrossing diagonals.

\begin{Definition}\cite[Definition 7.1]{CH19}
   For $m \in \mathbb{Z}_{\geq 2}$ let $T$ be a triangulation of a convex $m$-gon. A labeling of $T$ is an assignment of integers $a_t$ to each triangle $t$. Let $d$ be the sum of the number of negative labels and half the number of $0$ labels.
   We call a labeling admissible if it satisfies the following two conditions.
   \begin{itemize}
       \item The set of triangles $t$ with $a_t\notin\{\pm1\}$ can be partitioned into two element sets $\{t_1,t_2\}$ such that $t_1$ and $t_2$ have a common edge and $a_{t_1}=-a_{t_2}$.
       \item $d$ is even.
   \end{itemize}
\end{Definition}
The first condition ensures that $d$ is an integer because triangles labeled $0$ only show up in pairs.\\ 
We will sometimes need a bit more flexibility. Therefore, we generalize the definition of admissible labelings.
\newpage
\begin{Definition}
    Let $L$ be a labeling of a triangulation $T$. Let $\varepsilon \in \{\pm 1\}$ We call $L$ $\varepsilon$-admissible if it satisfies the following conditions.
    \begin{itemize}
       \item The set of triangles $t$ with $a_t\notin\{\pm1\}$ can be partitioned into two element sets $\{t_1,t_2\}$ such that $t_1$ and $t_2$ have a common edge and $a_{t_1}=-a_{t_2}$.
       \item $(-1)^d=\varepsilon$.
   \end{itemize}
\end{Definition}

Therefore, $1$-admissible is the same as admissible. The $-1$-admissible labelings are labelings that satisfy the first condition of admissible labelings but fail to satisfy the second one.

\begin{Definition}
    Let $L$ be an admissible labeling. We call a tuple $\{t_1,t_2\}$ a  quadrilateral  
    if $t_1$ are $t_2$ are adjacent triangles in the underlying triangulation of $L$.
    If $t_1$ is labeled $a$ and $t_2$ is labeled $-a$ for some $a \in \mathbb{Z}$, we call $\{t_1,t_2\}$ a polarized  quadrilateral of the labeling.
\end{Definition}

\begin{Example}
Below we see four labelings with the same underlying triangulation. The first labeling (from left to right, top to bottom) is admissible. The second labeling is $-1$-admissible. The third labeling is neither admissable nor $-1$-admissible because there is a label $3$ but no label $-3$. The fourth labeling is also neither admissable nor $-1$-admissible because the triangles labeled $3$ and $-3$ have no common edge.
    \begin{figure}[H]
        \caption{$\pm1$-admissible labelings}
        \centering
        \vspace*{0.25in}
        \includegraphics[width=0.8\linewidth]{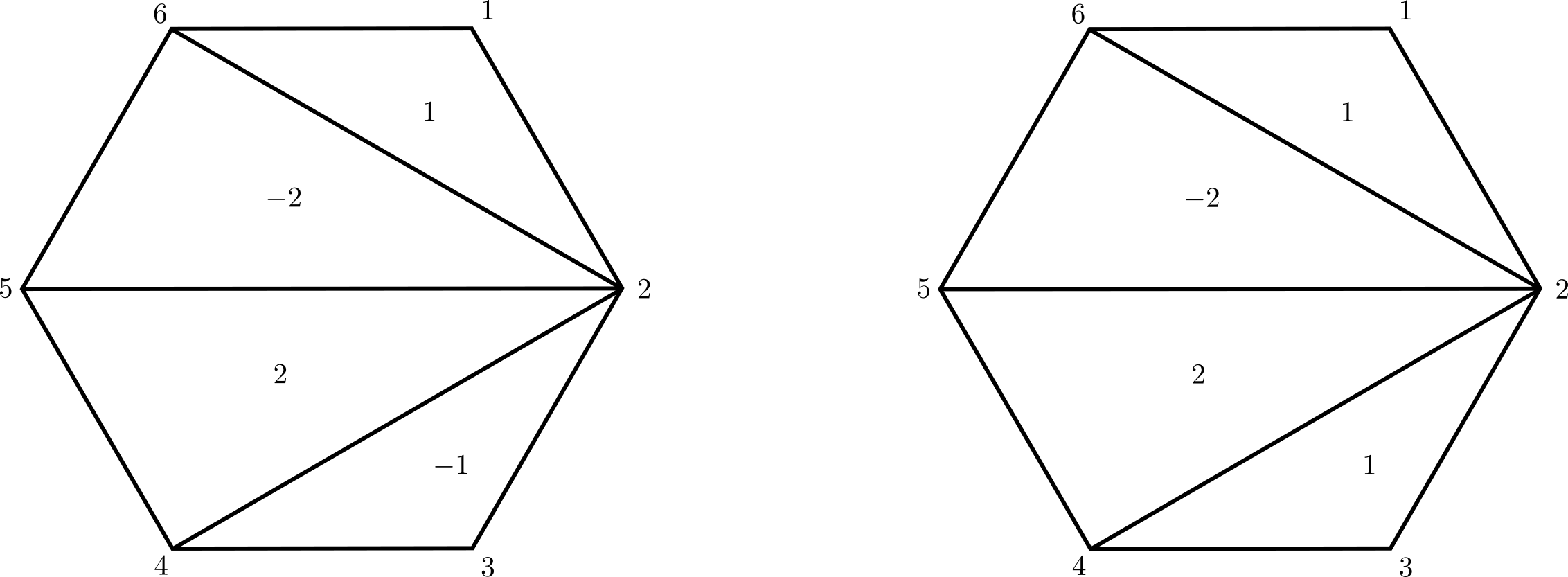}
        \vspace*{0.25in}
        \label{fig:admissible Labelings}
        \caption{non admissible labelings}
        \vspace*{0.25in}
         \includegraphics[width=0.8\linewidth]{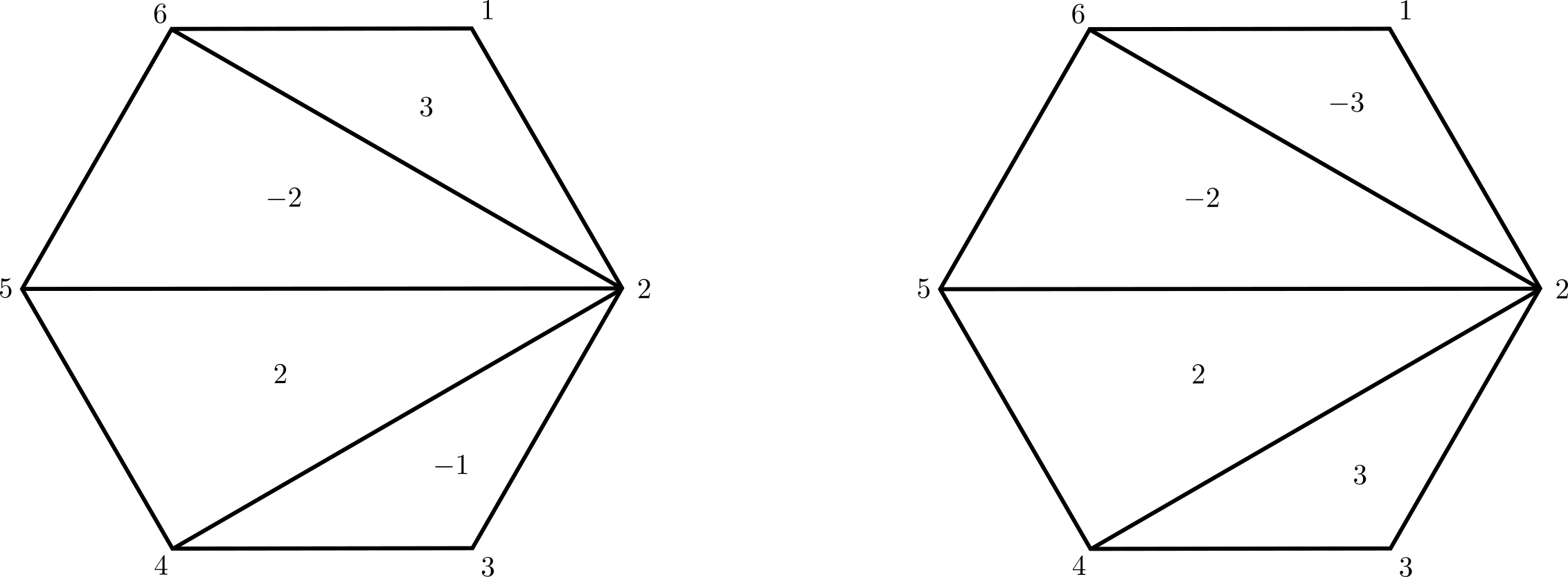}
        \label{fig:Non admissible Labeling 1}
       
    \end{figure}
\end{Example}
\newpage
\begin{Theorem}\cite[Theorem 7.5]{CH19}\label{quiddity cycle from labeling}
  \begin{enumerate}[label=(\alph*)]
      \item Let $T$ be a triangulation of an $m$-gon for some $m \in \mathbb{Z}_{\geq 2}$ with vertices labeled $1,\dots,m$ (in clockwise order). Let $L$ be an admissible labeling of $T$. For $i=1,\dots,m$ let $c_{i}$ be the sum of the labels of triangles adjacent to the vertex $i$. Then $(c_1,\dots,c_m)$ is the quiddity cycle of a tame frieze pattern over $\mathbb{Z}$.
      \item Every quiddity cycle $(c_1,\dots,c_m)$ can be obtained from an admissible labeling in this way.
  \end{enumerate}
\end{Theorem}
\begin{Remark}\label{extension to 1-labelings}
    Theorem \ref{quiddity cycle from labeling} can be adapted to $-1$-admissible labelings. If we perform the same steps, we will obtain a $1$-quiddity cycle. Moreover, each $1$-quiddity cycle can be obtained this way.
\end{Remark}

\begin{Definition}
    Let $n \in \mathbb{Z}_{\geq 0}$. Let $L$ be a admissible labeling of an $n+3$-gon. By Lemma \ref{quiddity cycle from labeling}, this labeling will yield a quiddity cycle $(c_1,\dots,c_{n+3})$. We write $f_n$ for the map sending an admissible labeling to the tame frieze pattern corresponding to the quiddity cycle obtained this way.
\end{Definition}

\begin{Remark}\label{map remark}
    Because of Theorem \ref{quiddity cycle from labeling} b), the map $f_n$ is surjective. The map is not injective \cite[Remark 7.4 (5)]{CH19}. If we restrict the map to admissible labelings where every triangle is labeled $1$, we recover the bijection from \cite{CC73a}.
\end{Remark}

\label{section_tame}
\section{Wild frieze patterns}

In this section we consider frieze patterns that are allowed to have wild entries.

\begin{Definition}
    Let $\mathcal{F}$ be a frieze pattern of height $n$ with entries $x_{i,j}$.\\ We write $\rho_i(\mathcal{F}):=(0,1,x_{i,i+2},\dots,x_{i,i+n+1},1,0)$ for the $i$-th row of $\mathcal{F}$.
\end{Definition}

\begin{Proposition}\label{wild frieze patterns are sequences of tame frieze patterns}
    There is a bijection between frieze patterns of height $n$ over $\mathbb{Z}$ and bi-infinite sequences $(\mathcal{F}_k)_{k\in \mathbb{Z}}$ of tame frieze patterns of height $n$ over $\mathbb{Z}$ such that $\rho_{k+1}(\mathcal{F}_k)=\rho_{k+1}(\mathcal{F}_{k+1})$ for all $k \in \mathbb{Z}$.

    \begin{proof}
        Let $\mathcal{F}$ be a frieze pattern of height with entries $x_{i,j}$ in $\mathbb{Z}$. We construct a sequence of tame frieze patterns as follows: For $k\in \mathbb{Z}$ consider two consecutive rows
        \begin{align*}
        \begin{NiceMatrixBlock}[auto-columns-width]
          \begin{NiceMatrix}
         0&1&x_{k,k+2}&\dots&x_{k,k+n+1}&1&0  \\
         &0&1&x_{k+1,k+3}&\dots&x_{k+1,k+n+2}&1&0.
   \end{NiceMatrix} 
    \end{NiceMatrixBlock}
    \end{align*}
By Lemma \ref{two rows determine tame frieze}, there is a tame frieze pattern $\mathcal{T}^{\mathcal{F}}_k$ where this segment appears as the $k$-th and $k+1$st row.
Moreover, it follows from the construction that $\rho_{k+1}(\mathcal{T}^{\mathcal{F}}_k)=\rho_{k+1}(\mathcal{T}^{\mathcal{F}}_{k+1})$.
\\
\\
If we start with a sequence $(\mathcal{G}_k)_{k\in \mathbb{Z}}$ of tame frieze patterns over $\mathbb{Z}$ satisfying $\rho_{k+1}(\mathcal{G}_k)=\rho_{k+1}(\mathcal{G}_{k+1})$ for all $k \in \mathbb{Z}$, we can get a frieze pattern $\mathcal{W}((\mathcal{G}_k)_{k\in \mathbb{Z}})$ by choosing the $k$-th row of $\mathcal{G}_k$ as the $k$-th row of $\mathcal{W}((\mathcal{G}_k)_{k\in \mathbb{Z}})$.

Since $\rho_{k+1}(\mathcal{G}_k)=\rho_{k+1}(\mathcal{G}_{k+1})$ and $\mathcal{G}_{k+1}$ satisfies the $SL_2$ condition, each $2\times2$ determinant 
\begin{align*}
    \begin{vmatrix}
     x_{i,j}&x_{i,j+1}  \\
     x_{i+1,j}&x_{i+1,j+1} 
\end{vmatrix}
\end{align*}
 equals $1$.
\\
We can now check that these two constructions are inverse to each other.
First, let $(\mathcal{G}_k)_{k\in \mathbb{Z}}$ be a sequence of tame frieze patterns with $\rho_{k+1}(\mathcal{G}_k)=\rho_{k+1}(\mathcal{G}_{k+1})$ for all $k \in \mathbb{Z}$.

The $k$-th and $k+1$st row of $\mathcal{W}((\mathcal{G}_k)_{k\in \mathbb{Z}}$ are the same as in $\mathcal{G}_k$. Since $\mathcal{T}_k^{\mathcal{W}((\mathcal{G}_k)_{k\in \mathbb{Z}})}$ is defined by these two rows, Lemma \ref{two rows determine tame frieze} shows $\mathcal{T}_k^{\mathcal{W}((\mathcal{G}_k)_{k\in \mathbb{Z}})}=\mathcal{G}_k$ for all $k \in \mathbb{Z}$ and the two bi-infinite sequences agree.
\\
\\
Conversely, let $\mathcal{F}$ be a frieze pattern of height $n$ over $\mathbb{Z}$.
Then the $k$-th row of $\mathcal{T}_k^{\mathcal{F}}$ is the same as the $k$-th row of $\mathcal{F}$. This row will also become the $k$-th row of $\mathcal{W}((\mathcal{T}_k^{\mathcal{F}})_{k\in \mathbb{Z}})$. So $\mathcal{W}((\mathcal{T}_k^{\mathcal{F}})_{k\in \mathbb{Z}})=\mathcal{F}$.
    \end{proof}
\end{Proposition}

\begin{Remark}
    If $\mathcal{F}$ is already tame, then $\mathcal{T}^{\mathcal{F}}_k=\mathcal{F}$ for every $k \in \mathbb{Z}$. 
    If $\mathcal{T}^{\mathcal{F}}_k=\mathcal{T}^{\mathcal{F}}_{k+1}=\dots=\mathcal{T}^{\mathcal{F}}_{k+l}$, then the frieze pattern $\mathcal{F}$ is tame between row $k+1$ and $k+l$. If $\mathcal{T}^{\mathcal{F}}_k\neq \mathcal{T}^{\mathcal{F}}_{k+1}$, then there is some wild entry in row $k+1$ of $\mathcal{F}$.
\end{Remark}

Proposition \ref{wild frieze patterns are sequences of tame frieze patterns} basically tells us that we can glue together tame frieze patterns with common rows to get wild ones. The natural next step is to characterize frieze patterns that share a row. In other words we would like to understand the preimages of $\rho_i$. Since tame frieze patterns can be rotated, it suffices to consider tame frieze patterns that have a common first row.

\begin{Proposition}\label{common row Lemma}
Let $\mathcal{T}$ and $\tilde{\mathcal{T}}$ be two tame frieze patterns over $\mathbb{Z}$. Then $\rho_1(\mathcal{T})=\rho_1(\tilde{\mathcal{T}})$ if and only if there are $r \in \mathbb{Z}_{\geq 0}$, $\varepsilon_i \in \{\pm 1\}$ with $\prod_{i=0}^{r}\varepsilon_i=-1$, $a_i,b_i \in \mathbb{Z}$ with 
$\sum^r_{i=0}a_i=\sum^r_{i=0}b_i=0$ and $\varepsilon_i$ quiddity cycles $(c_{i,1},\dots,c_{i,m_i})$ for $i=0,\dots,r$ 
such that 
\begin{align*}
    (c_{0,1},\dots,c_{0,m_0-1},c_{0,m_0}+a_0,\dots\dots,c_{r,1},\dots,c_{r,m_r-1},c_{r,m_r}+a_{r})
\end{align*}
\normalsize
and 
\begin{align*}
    (c_{0,1},\dots,c_{0,m_0-1},c_{0,m_0}+b_0,\dots\dots,c_{r,1},\dots,c_{r,m_r-1},c_{r,m_r}+b_{r})
\end{align*}
are the quiddity cycles of $\mathcal{T}$ and $\tilde{\mathcal{T}}$ respectively.
\begin{proof}
First, assume that $\mathcal{T}$ and $\tilde{\mathcal{T}}$ have the same first row.
    Because of the $SL_2$ rule, a frieze pattern over $\mathbb{Z}$ looks like 
    \begin{align*}
        \begin{NiceMatrix}
            &\pm1&\\
            \mp1&0&\pm1\\
            &\mp1&&
        \end{NiceMatrix}
    \end{align*} 
    around each $0$.
    Therefore, we can write the common first line as
    \begin{align*}
        \rho_1(\mathcal{T})=\rho_1(\tilde{\mathcal{T}})=(0,1,x_{0,1},\dots,x_{0,n_0},-\mu_1,0,\mu_1,\dots,-\mu_r,0,\mu_r,x_{r,1},\dots,x_{r,n_r},1,0) 
    \end{align*}
    with $\mu_i \in \{\pm1\}$ and none of the $x_{i,j}$ equal to $0$.
    Now consider a third tame frieze pattern with the same first row and with the second row
    \begin{align*}
        (0,1,y_{0,1},\dots,y_{0,n_0},-\mu_1,0,\mu_1,\dots,-\mu_r,0.\mu_r,y_{r,1},\dots,y_{r,n_r},1,0).
    \end{align*}
    The $y_{i,j}$ are uniquely determined by the $x_{i,j}$ and the $\mu_i$ through the $SL_2$ rule.\\
    
    We can now excise parts of this third tame frieze pattern to obtain infinite arrays
    \begin{align}\label{excised part}
        \begin{NiceMatrix}
            \ddots&&&&&\ddots\\
            0&\mu_i&x_{i,1}&\dots&x_{i,n_i}&-\mu_{i+1}&0\\
            &0&\mu_i&y_{i,1}&\dots&y_{i,n_i}&-\mu_{i+1}&0\\
            &&&\ddots&&&&\ddots
        \end{NiceMatrix}.
    \end{align}
    If we set $\mu_0=-\mu_{r+1}=1$, this also accounts for the left- and rightmost part.
    These excised parts can be extended similarly as in Lemma \ref{two rows determine tame frieze}. 
    After multiplying every entry by $\mu_i$, these will either be a frieze pattern or of the form \eqref{-1 frieze pattern}, depending on whether $\mu_i=\mu_{i+1}$ or not.
    This corresponds to an $\varepsilon_i:=\mu_{i}\mu_{i+1}$-quiddity cycle.
    Since
    \begin{align*}
        \frac{x_{i,j-1}+x_{i,j+1}}{x_{i,j}}=\frac{\mu_ix_{i,j-1}+\mu_ix_{i,j+1}}{\mu_ix_{i,j}}
    \end{align*}
    (and similarly for the second row), we can now use Lemma \ref{quiddity from rows} to write the quiddity cycle of the third tame frieze pattern as 
    \begin{align*}
    (c_{0,1},\dots,c_{0,m_0},c_{1,1},\dots,c_{1,m_1},\dots\dots,c_{r-1,1},\dots,c_{r-1,m_{r-1}},c_{r,1},\dots,c_{r,m_r})
\end{align*}
where $(c_{i,1},\dots,c_{i,m_i})$ are the $\mu_i\mu_{i+1}$ quiddity cycles corresponding to the excised parts \eqref{excised part}. Now, as long as $x_{i,j+1}$ is not $0$, by Lemma \ref{quiddity from rows} the $j$-th entries of the quiddity cycles of $\mathcal{T}$ and $\tilde{\mathcal{T}}$ are both the same as the $j$-th entry of the third quiddity cycle. Now we need to understand what happens at the zeroes. We look at the parts of $\mathcal{T}$ between two zeroes in the first row and the part of the second row underneath it.

\begin{align*}
        \begin{NiceMatrix}
            \ddots&&&&&\ddots\\
            0&\mu_i&x_{i,1}&\dots&x_{i,n_i}&-\mu_{i+1}&0\\
            -\mu_i&w_{i,-1}&w_{i,0}&w_{i,1}&\dots&w_{i,n_i}&-\mu_{i+1}&w_{i+1,-1}\\
            &&&\ddots&&&&\ddots
        \end{NiceMatrix}
    \end{align*}

We claim that we always have $w_{i,j}=y_{i,j}+\mu_ix_{i,j+1}w_{i,-1}$. \\
(We set $x_{i,0}:=\mu_i,x_{i,n_i+1}:=-\mu_{i+1},y_{i,0}:=\mu_i,y_{i,-1}:=0$.)
In order to prove this claim we proceed by induction on $j$.
Indeed this holds for $j=-1$ because $y_{i,-1}=0$ and $x_{i,0}=\mu_i=\mu_i^{-1}$. Now assume that it holds for some $j\in \{-1,0,1,\dots,n_i-1\}$. Then we have \begin{align*}
    \begin{vmatrix}
        x_{i,j+1}&x_{i,j+2} \\
        x_{i,j+1}\mu_{i}w_{i,-1}+y_{i,j}&w_{i,j+1}
    \end{vmatrix}=
    \begin{vmatrix}
        x_{i,j+1}&x_{i,j+2} \\
        y_{i,j}&y_{i,j+1}
    \end{vmatrix}&=1
\end{align*}
by the $SL_2$ rule.
Thus, we get
\begin{align*}
    \begin{vmatrix}
        x_{i,j+1}&x_{i,j+2} \\
        x_{i,j+1}\mu_{i}w_{i,-1}&w_{i,j+1}-y_{i,j+1}
    \end{vmatrix}&=0
\end{align*}
from the multilinearity of the determinant.
Since $x_{i,j+1}\neq 0$, the second row needs to be a scalar multiple of the first. 
It follows that \begin{align*}
    w_{i,-1}\mu_i\left(\begin{NiceMatrix}
        x_{i,j+1}&x_{i,j+2}
    \end{NiceMatrix}\right)
    =\left( \begin{NiceMatrix}x_{i,j+1}\mu_{i}w_{i,-1}&w_{i,j+1}-y_{i,j+1}\end{NiceMatrix}\right)
\end{align*}
and therefore $w_{i,j+1}=y_{i,j+1}+x_{i,j+2}\mu_iw_{i,-1}$, proving the claim.\\
In particular, this tells us 
\begin{equation}\label{w in terms of y}
    w_{i,n_{i}}=y_{i,n_i}-\mu_i\mu_{i+1}w_{i,-1}.
\end{equation}
 
Now we can use this to compute the remaining entries of the quiddity cycle. By Lemma \ref{quiddity from rows}, these are given by \begin{align*}
\frac{w_{i+1,-1}+w_{i,n_i}}{-\mu_{i+1}} \overset{(\ref{w in terms of y})}{=}    \frac{w_{i+1,-1}+y_{i,n_i}-\mu_i\mu_{i+1}w_{i,-1}}{-\mu_{i+1}}=-\mu_{i+1}y_{i,n_{i}}-\mu_{i+1}w_{i+1,-1}+\mu_{i}w_{i,-1}.
\end{align*}

At the boundaries of the frieze pattern we always have $w_{0,-1}=0$ and $w_{r+1,-1}=0$.
We see that $c_{i,m_i}=-\mu_{i+1}y_{i,n_i}$ for $i=0,\dots,r$ by applying Lemma \ref{quiddity from rows} to the third tame frieze pattern.
This means if we define $a_0:=-\mu_1w_{1,-1}$, $a_r:=\mu_{r}w_{r,-1}$, $a_i:=\mu_iw_{i,-1}-\mu_{i+1}w_{i+1,-1}$ for $1\leq i \leq r-1$, we have $\sum^r_{i=0}a_i=0$ and the missing quiddity cycle entries are given by $c_{i,m_i}+a_i$. In the same way we find $b_i$ with $\sum^r_{i=0}b_i=0$ such that the missing quiddity cycle entries of $\tilde{\mathcal{T}}$ are $c_{i,m_i}+b_i$. Finally, we get \begin{align*}
    \prod_{i=0}^r\varepsilon_i=\mu_0\left(\prod_{i=1}^{r}\mu_{i}^{2}\right)\mu_{r+1}=-1.
\end{align*}
\\
\\
Now assume that there are $r \in \mathbb{Z}_{\geq 0}$, $\varepsilon_i \in \{\pm 1\}, a_i,b_i$ with 
$\sum^r_{i=0}a_i=\sum^r_{i=0}b_i=0$ and $\varepsilon_i$ quiddity cycles $(c_{i,1},\dots,c_{i,m_i})$ for $i=0,\dots,r$ 
such that 
\begin{align*}
    (c_{0,1},\dots,c_{0,m_0-1},c_{0,m_0}+a_0,\dots\dots,c_{r,1},\dots,c_{r,m_r-1},c_{r,m_r}+a_{r})
\end{align*}
is the quiddity cycle of $\mathcal{T}$
and 
\begin{align*}
    (c_{0,1},\dots,c_{0,m_0-1},c_{0,m_0}+b_0,\dots\dots,c_{r,1},\dots,c_{r,m_r-1},c_{r,m_r}+b_{r})
\end{align*}
is the quiddity cycle of $\tilde{\mathcal{T}}$. We want to show that the frieze patterns have the same first row.
We denote the entries of $\mathcal{T}$ by $x_{i,j}$ and the entries of $\tilde{\mathcal{T}}$ by $\tilde{x}_{i,j}$.

Let $s_k:=m_0+\dots+m_k$. Then we have $c_{0,j}:=x_{j,j+2}$ for $j=1,\dots,m_0-1$ and $c_{i,j}=x_{s_{i-1}+j,s_{i-1}+j+2}$ for $i=1,\dots,r$, $1\leq j \leq m_i-1$. Analogously, we have $c_{0,j}:=\tilde{x}_{j,j+2}$ for $j=1,\dots,m_0-1$ and $c_{i,j}=\tilde{x}_{s_{i-1}+j,s_{i-1}+j+2}$ for $i=1,\dots,r$, $1\leq j \leq m_i-1$.\\
\\
We first show that $x_{1,j}=\tilde{x}_{1,j}=-\varepsilon_0\cdots\varepsilon_k$ if $j=s_k$, $x_{1,j}=\tilde{x}_{1,j}=0$ if $j=s_k+1$ and $x_{1,j}=\tilde{x}_{1,j}=\varepsilon_0\cdots\varepsilon_k$ if $j=s_k+2$.
We proceed by induction on $k$ for this. 
If we rephrase the recursion \eqref{recursion} in terms of $\eta$-matrices, as done in the proof of \cite[Proposition 2.4]{CH19}, we obtain

\begin{align*}
    \left(\begin{NiceMatrix}
        x_{1,j+1}&-x_{1,j}
    \end{NiceMatrix}\right)\cdot \prod_{k=j}^{l}\eta(x_{k,k+2})=  \left(\begin{NiceMatrix}
        x_{1,l+2}&-x_{1,l+1}
    \end{NiceMatrix}\right).
\end{align*}
\\

Since $\prod_{j=1}^{m_i}\eta(c_{i,j})= \left(\begin{matrix}
    \varepsilon_i&0\\
    0&\varepsilon_i
\end{matrix}\right)$, we have \begin{align}\label{incomplete product}
    \prod_{j=1}^{m_i-1}\eta(c_{i,j})=\varepsilon_i\eta(c_{i,m_i})^{-1}=\left(
    \begin{NiceMatrix}
        0&\varepsilon_i\\
        -\varepsilon_i&\varepsilon_ic_{i,m_i}
    \end{NiceMatrix}
    \right)
\end{align}
for $i=0,\dots,r.$
\\
\\
This tells us
\begin{align*}
    \left(\begin{NiceMatrix}
        x_{1,m_0+1}&-x_{1,m_0}
    \end{NiceMatrix}\right)=\left(\begin{NiceMatrix}
        1&0
    \end{NiceMatrix}\right)\prod_{j=1}^{m_0-1}\eta(x_{j,j+2})=
    \left(\begin{NiceMatrix}
        1&0
    \end{NiceMatrix}\right)\prod_{j=1}^{m_0-1}\eta(c_{0,j})\overset{\eqref{incomplete product}}{=}
    \left(\begin{NiceMatrix}
        0&\varepsilon_0
    \end{NiceMatrix}\right).
\end{align*}
\\
\\
Moreover,
\begin{align*}
    \left(\begin{NiceMatrix}
        x_{1,m_0+2}&-x_{1,m_0+1}
    \end{NiceMatrix}\right)=
    \left(\begin{NiceMatrix}
        x_{1,m_0+1}&-x_{1,m_0}
    \end{NiceMatrix}\right) \eta(x_{m_0,m_0+2}+a_0)
    =
    \left(\begin{NiceMatrix}
        0&\varepsilon_0
    \end{NiceMatrix}\right)\eta(x_{m_0,m_0+2}+a_0)=
    \left( \begin{NiceMatrix}
        \varepsilon_0&0    
    \end{NiceMatrix} \right).
\end{align*}
\\
\\
Therefore, $x_{1,m_0}=-\varepsilon_0$, $x_{1,m_0+1}=0$ and $x_{1,m_0+2}=\varepsilon_0$.
\\ \\ \\
Assuming $ \left( \begin{NiceMatrix}
        x_{1,s_k}&x_{1,s_k+1}&x_{1,s_k+2}    
    \end{NiceMatrix} \right)=\left( \begin{NiceMatrix}
        -\varepsilon_0 \cdots \varepsilon_k&0&\varepsilon_0\cdots \varepsilon_k   
    \end{NiceMatrix} \right)$,
we get
\begin{align*}
    \left(\begin{NiceMatrix}
        x_{1,s_{k+1}+1}&-x_{1,s_{k+1}}
    \end{NiceMatrix}\right)=\left(\begin{NiceMatrix}
        \varepsilon_0\cdots\varepsilon_k&0
    \end{NiceMatrix}\right)\prod_{j=s_k+1}^{s_{k+1}-1}\eta(x_{j,j+2})
    =\left(\begin{NiceMatrix}
        \varepsilon_0\cdots\varepsilon_k&0
    \end{NiceMatrix}\right)\prod_{j=1}^{m_{k+1}-1}\eta(c_{k+1,j})
    \overset{\eqref{incomplete product}}{=}\left(\begin{NiceMatrix}
        0&\varepsilon_0\cdots \varepsilon_{k+1}
    \end{NiceMatrix}\right)
\end{align*}
and 
\begin{align*}
    \left(\begin{NiceMatrix}
        x_{1,s_{k+1}+2}&-x_{1,s_{k+1}+1}
    \end{NiceMatrix}\right)=\left(\begin{NiceMatrix}
        0&\varepsilon_0\cdots\varepsilon_{k+1}
    \end{NiceMatrix}\right)\eta(x_{s_{k+1},s_{k+1}+2}+a_{k+1})=\left(\begin{NiceMatrix}
        \varepsilon_0\cdots \varepsilon_{k+1}&0
    \end{NiceMatrix}\right).
\end{align*}
\\
\\
 We can apply the same argument to the entries $\tilde{x}_{1,s_k}$,$\tilde{x}_{1,s_k+1}$ and $\tilde{x}_{1,s_k+2}$ of $\tilde{\mathcal{T}}$.
\\
\\
Now we are ready to show that the remaining entries of the first rows also coincide.
We again use induction for this. The first two entries are $0$ and $1$ for both frieze patterns. Now assume $x_{1,j}=\tilde{x}_{1,j}$ for all $j<l$. If $l\notin \{s_0+2,s_1+2,\dots,s_r+2\}$, the quiddity cycle entries $x_{l-2,l}$ and $\tilde{x}_{l-2,l}$ are the same.
The recursions $x_{1,l}=x_{l-2,l}x_{1,l-1}-x_{1,l-2}$ and $\tilde{x}_{1,l}=\tilde{x}_{l-2,l}\tilde{x}_{1,l-1}-\tilde{x}_{1,l-2}$ now show $x_{1,l}=\tilde{x}_{1,l}$. 
Otherwise, if $l=s_k+2$ for some $k$, we already know that $x_{1,l}=\varepsilon_0\cdots\varepsilon_k$.
\end{proof}
\end{Proposition}

\begin{Remark}\label{cluster interpretation}
    Tame frieze patterns over $\mathbb{Z}$ are sometimes thought of as ring homomorphisms from a type $A_n$ cluster algebra to $\mathbb{Z}$.
    Each row of a frieze pattern then contains the values of the cluster variables contained in a certain cluster.
    Thus, by Proposition \ref{wild frieze patterns are sequences of tame frieze patterns} we can view a (possibly wild) frieze pattern as a bi-infinite sequence $(\varphi_k)_{k\in \mathbb{Z}}$ of such ring homomorphisms such that $\varphi_k$ and $\varphi_{k+1}$ restrict to the same map on the cluster corresponding to the $k+1$st row.  
    However, not every cluster corresponds to a row of the frieze pattern. More generally there is a cluster for each ``lightning bolt" of entries. (Even this does not give us all the clusters.)
    Hence, it would be interesting to have a version of Proposition \ref{common row Lemma} for pairs of tame frieze patterns that agree on a ``lightning bolt". (See for example \cite{pressland2020frieze} for an overview of the connections between frieze patterns and cluster algebras.)

\end{Remark}

\begin{Corollary} \label{common row same sum}
    If two tame frieze patterns with entries in $\mathbb{Z}$ have a common row, the sum of their quiddity cycle entries will be the same.
    \begin{proof}
        We can rotate the frieze patterns and the quiddity cycles so that we can assume they have the same first row. Then Proposition \ref{common row Lemma} tells us that the quiddity cycles are both of the form \\$(c_{0,1},\dots,c_{0,m_0-1},c_{0,m_0}+a_1,\dots\dots,c_{r,1},\dots,c_{r,m_r-1},c_{r,m_r}+a_{r})$ for certain $a_i$ with $\sum_{i=1}^ra_i=0$. Therefore, the sum is \begin{align*}
            \sum_{i=0}^r\sum_{j=1}^{m_i}c_{i,j}.
        \end{align*}
        for both frieze patterns.
    \end{proof}
\end{Corollary}

\begin{Remark}
    This Corollary, and therefore also Proposition \ref{common row Lemma}, are not true anymore if we replace $\mathbb{Z}$ with an arbitrary commutative ring.
    For example, the tame frieze patterns over $\mathbb{Q}$
    \begin{align*}
     \begin{NiceMatrixBlock}[auto-columns-width]
        \begin{NiceMatrix}
            \ddots&&&&&&\ddots\\
            &0&1&\frac{1}{2}&0&-\frac{1}{2}&1&0\\
            &&0&1&2&0&-2&1&0\\
            &&&0&1&\frac{1}{2}&-2&\frac{1}{2}&1&0\\
            &&&&0&1&-2&0&2&1&0\\
            &&&&&0&1&-\frac{1}{2}&0&\frac{1}{2}&1&0\\
            &&&&&&0&1&-2&-2&-2&1&0\\
            &&&&&&&\ddots&&&&&&\ddots
        \end{NiceMatrix}
        \end{NiceMatrixBlock}
    \end{align*}
    and 
    \begin{align*}
        \begin{NiceMatrixBlock}[auto-columns-width]
        \begin{NiceMatrix}
            \ddots&&&&&&\ddots\\
            &0&1&\frac{1}{2}&0&-\frac{1}{2}&1&0\\
            &&0&1&2&-1&0&1&0\\
            &&&0&1&0&-1&\frac{1}{2}&1&0\\
            &&&&0&1&-2&0&2&1&0\\
            &&&&&0&1&-\frac{1}{2}&-1&0&1&0\\
            &&&&&&0&1&0&-1&-2&1&0\\
            &&&&&&&\ddots&&&&&&\ddots
        \end{NiceMatrix}
        \end{NiceMatrixBlock}
    \end{align*}
    have the same first row, but the quiddity cycle entries of the first one sum up to $-\frac{3}{2}$, while the entries of the second one sum up to $0$.\\
    However, the Lemma should still hold for a commutative ring $R$ where $1$ and $-1$ are the only units. This ensures that the entries next to a $0$ have to be $\pm1$. This is crucial for the proof of Proposition \ref{common row Lemma}.
\end{Remark}

\begin{Definition}
    
  Let $\mathcal{F}$ be a frieze pattern over $\mathbb{Z}$. Let $\mathcal{T}_k$ the tame frieze pattern obtained from row $k$ and $k+1$ of $\mathcal{F}$. Let $(c_1,\dots,c_m)$ be the quiddity cycle of $\mathcal{T}_k$. We define  $Q(\mathcal{F}):=\sum_{i=1}^{m}c_i$.\\
  $Q_k(\mathcal{F})$ does not depend on $k$: Let $\mathcal{T}_{k+1}$ be the tame frieze pattern obtained from rows $k+1$ and $k+2$ of $\mathcal{F}$. Then $\mathcal{T}_k$ and $\mathcal{T}_{k+1}$ have the same $k+1$st row. Therefore, $Q_k(\mathcal{F})=Q_{k+1}(\mathcal{F})$ by Corollary \ref{common row same sum} and we continue by induction. Hence, we can simply write $Q(\mathcal{F}):=Q_1(\mathcal{F})$. We call $Q(\mathcal{F})$ the quiddity number of $\mathcal{F}$.
\end{Definition}

\begin{Example}
    Let $\mathcal{F}$ be the frieze pattern obtained by repeating the segment \eqref{Wild Example}.
    We want to calculate $Q(\mathcal{F})$. The first two rows uniquely define the tame frieze pattern

    \begin{align}
     \begin{NiceMatrixBlock}[auto-columns-width]
        \begin{NiceMatrix}
            \ddots&&&&&\ddots\\
            0&1&3&-1&-3&-2&1&0\\
            &0&1&0&-1&-1&0&1&0\\
            &&0&1&0&-1&-1&3&1&0\\
            &&&0&1&1&0&-1&0&1&0\\
            &&&&0&1&1&-3&-1&0&1&0\\
            &&&&&0&1&-2&-1&-1&1&1&0\\
            &&&&&&0&1&0&-1&0&1&1&0\\
            &&&&&&&& \ddots&&&&&\ddots\\
        \end{NiceMatrix}
    \end{NiceMatrixBlock}
    \end{align}
    Summing up the entries of the quiddity cycle $(3,0,0,1,1,-2,0)$ shows $Q(\mathcal{F})=3$.
\end{Example}

It was a rather arbitrary choice that we considered pairs of consecutive rows and not pairs of consecutive columns. Using pairs of columns, a wild frieze pattern gives rise to a second sequence of tame frieze patterns. This is in general not the same sequence. However, if we compute the sum of the quiddity cycle entries of those tame frieze patterns, we will still get $Q(\mathcal{F})$ and not a second invariant.

\begin{Lemma}
    Let $\mathcal{F}$ be a frieze pattern over $\mathbb{Z}$. Let $\tilde{Q}(\mathcal{F})$ be the number obtained by replacing rows by columns everywhere in the definition of $Q(\mathcal{F})$.
    Then $Q(\mathcal{F})=\tilde{Q}(\mathcal{F})$.
    \begin{proof}
        Consider the tame frieze pattern $\mathcal{T}^{\mathcal{F}}_0$ as in the proof of Lemma \ref{wild frieze patterns are sequences of tame frieze patterns}. Let $\mathcal{F}_0$ be the frieze pattern defined by \begin{align*}
            \rho_i(\mathcal{F}_0):=\begin{cases}
            \rho_i(\mathcal{F})& \text{ if } i\leq 0 \\
            \rho_i(\mathcal{T}^{\mathcal{F}}_0)& \text{ if } i\geq 1
        \end{cases}.
        \end{align*}\\
        We have $Q(\mathcal{F})=Q(\mathcal{F}_0)$ and $\tilde{Q}(\mathcal{F})=\tilde{Q}(\mathcal{F}_0)$ because $\mathcal{F}$ and $\mathcal{F}_0$ have common rows and common columns. Similarly, we get $Q(\mathcal{F}_0)=Q(\mathcal{T}^{\mathcal{F}}_0)$ and $\tilde{Q}(\mathcal{F}_0)=\tilde{Q}(\mathcal{T}^{\mathcal{F}}_0)$. But since $\mathcal{T}^{\mathcal{F}}_0$ is tame, $Q(\mathcal{T}^{\mathcal{F}}_0)$ and $\tilde{Q}(\mathcal{T}^{\mathcal{F}}_0)$ are both the sum of the entries of the quiddity cycle of $\mathcal{T}^{\mathcal{F}}_0$.

    \end{proof}
\end{Lemma}

Since any two tame frieze patterns that share a row have the same quiddity number, we can also define the quiddity number for rows.

\begin{Definition}
    Let $v:=(0,1,x_1,\dots,x_n,1,0)$ be a tuple that appears as a row in some tame frieze pattern $\mathcal{T}$ over $\mathbb{Z}$. Then we define $Q(v):=Q(\mathcal{T})$.
\end{Definition}

\begin{Definition} \label{definition Graph}
    Let $n \in \mathbb{Z}_{\geq 0}$. Consider the (infinite) directed graph $\Gamma_{2,n}(\mathbb{Z})$ where the vertices are all tuples $(0,1,x_1,\dots,x_n,1,0) \in \mathbb{Z}^{n+4}$ that appear as rows in some frieze pattern over $\mathbb{Z}$ and where an arrow $(0,1,x_1,\dots,x_n,1,0)\rightarrow(0,1,y_1,\dots,y_n,1,0)$ exists if 
    \begin{equation*}
           \begin{NiceMatrixBlock}[auto-columns-width]
        \begin{NiceMatrix}
        0&1&x_1&\dots&x_n&1&0\\
        &0&1&y_1&\dots&y_n&1&0
        \end{NiceMatrix}
    \end{NiceMatrixBlock}
    \end{equation*}
    satisfies the $SL_2$ rule everywhere.
    
    \end{Definition}
    
  This graph was first defined in \cite{Cun17} for $SL_k$ frieze patterns in a slightly different way. (The $2$ in $\Gamma_{2,n}(\mathbb{Z})$ indicates that we are looking at $SL_2$ frieze patterns.)
    Since all ($SL_2$) frieze patterns over $\mathbb{Z}_{\geq 1}$ are automatically tame, we allow entries in $\mathbb{Z}$. Moreover, the graph defined in \cite{Cun17} contains tuples which are not rows in any frieze pattern. This means we would get isolated vertices in the graph which are not useful for understanding wild frieze patterns.

    \begin{Remark}
    If a tuple appears as a row of a frieze pattern over $\mathbb{Z}$, then by Lemma \ref{wild frieze patterns are sequences of tame frieze patterns} it already appears as a row in a tame frieze pattern over $\mathbb{Z}$.\\
    A tuple $(0,1,x_1,\dots,x_n,1,0)$ appears as a row of a (tame) frieze pattern over $\mathbb{Z}$ if and only if for $i=1,\dots,n$ either $x_i=0,x_{i-1}=\pm1$ and $x_{i+1}=\mp1$ or $x_i \mid x_{i-1}+x_{i+1}$. 
    \\Indeed, if $0\neq x_i \nmid x_{i-1}+x_{i+1}$, then by Lemma \ref{quiddity from rows} there are non integral entries in the frieze pattern and if $x_i=0$, then the neighbouring entries need to be $1$ and $-1$.
    Conversely, if for $i=1,\dots,n$ either $x_i=0,x_{i-1}=\pm1$ and $x_{i+1}=\mp1$ or $x_i \mid x_{i-1}+x_{i+1}$, we can construct a second row $(1,y_1,\dots,y_n,1)$ recursively such that either $y_{i+1}:=\frac{x_{i+2}+x_i}{x_{i+1}}y_{i}-y_{i-1}$ if $x_{i+1} \neq 0$, or $y_{i+1}=0$ otherwise. Then clearly the entries of the second row are integers. 
    Moreover, if we define \begin{equation*}
        c_i:= \begin{cases}
        \frac{x_{i+2}+x_i}{x_{i+1}}, \text{ if }x_{i+1}\neq 0\\
        \frac{y_{i-1}}{y_{i}}, \text{ otherwise}
    \end{cases}
    \end{equation*} then the rows satisfy the recursions $x_{i+2}=c_ix_{i+1}-x_i$ and $y_{i+1}=c_iy_{i}-y_{i-1}$.\\
    (We have $y_i=\pm1 \neq 0$ in the definition of $c_i$ because $x_{i+1}=0$ and $x_i=\pm1$.)
    
    This way, we can compute
    \begin{align*}
        \begin{vmatrix}
            x_{i+1}&x_{i+2}\\
            y_{i}&y_{i+1}
        \end{vmatrix}=\begin{vmatrix}
            x_{i+1}&c_ix_{i+1}-x_{i}\\
            y_{i}&c_iy_{i}-y_{i-1}
        \end{vmatrix}
        =
        \begin{vmatrix}
            x_{i+1}&-x_i\\
            y_{i}&-y_{i-1}
        \end{vmatrix}
        =
        \begin{vmatrix}
            x_i&x_{i+1}\\
            y_{i-1}&y_{i}
        \end{vmatrix}.
    \end{align*}
    This can be used to show inductively that the two rows satisfy the $SL_2$ rule everywhere. Now it follows from Lemma \ref{two rows determine tame frieze} that there is a tame frieze pattern containing these rows.
    These observations show that the vertices of $\Gamma_{2,n}(\mathbb{Z})$ are the tuples $(0,1,x_1,\dots,x_n,1,0)\in \mathbb{Z}^{n+4}$ with either $x_i=0,x_{i-1}=\pm1$ and $x_{i+1}=\mp1$ or $x_i \mid x_{i-1}+x_{i+1}$ for all $i=1,\dots,n$.
    \end{Remark}
    
    Wild frieze patterns over $\mathbb{Z}$ correspond to infinite paths in $\Gamma_{2,n}(\mathbb{Z})$. Since arrows in $\Gamma_{2,n}(\mathbb{Z})$ are just pairs of consecutive rows, Lemma \ref{two rows determine tame frieze} tells us that each arrow corresponds to a tame frieze pattern.
    \\
    \\

\label{section_wild}
\section{\texorpdfstring{Strongly connected components of $\boldsymbol{\Gamma_{2,n}(\mathbb{Z})}$}{Strongly connected components}}

In this section we want to classify the strongly connected components of $\Gamma_{2,n}(\mathbb{Z})$ using the quiddity number. In particular, this enables us to count the number of strongly connected components.
More precisely, we want to show:
\begin{Theorem}\label{connected component classification}
    Let $v$ and $w$ be rows of frieze patterns of height $n$. Then $v$ and $w$ are in the same strongly connected component of $\Gamma_{2,n}(\mathbb{Z})$ if and only if one of the two following conditions holds:
    \begin{enumerate}
        \item $Q(v)=Q(w)=\pm (3n+3)$, and $v$ and $w$ are both rows of the same (twisted) Conway-Coxeter frieze pattern.
        \item $Q(v)=Q(w)\neq \pm(3n+3)$.
    \end{enumerate}
    
\end{Theorem}

We first deal with the easier case where $Q(v)=Q(w)=\pm(3n+3)$.

\begin{Lemma}\label{extreme quiddity number} 
    Let $v$ and $w$ be two rows of tame frieze patterns of height $n$ with 
    \begin{align*}
        Q(v)=Q(w)=\pm(3n+3).
    \end{align*} Then $v$ and $w$ are in the same strongly connected component  of $\Gamma_{2,n}(\mathbb{Z})$ if and only if they are part of the same (twisted) Conway-Coxeter frieze pattern.
    \begin{proof}
        If they are both part of the same (twisted) Conway-Coxeter frieze pattern, this frieze pattern yields a path from $v$ to $w$ and one from $w$ to $v$ because of the periodicity. 
        \\
        \\
        Conversely, assume that $v$ and $w$ are in the same strongly connected component.
        Let $\mathcal{T}$ be a tame frieze pattern containing $v$.  If $Q(v)=Q(\mathcal{T})=3n+3$, any admissible labeling that is mapped to $\mathcal{T}$ by $f_n$ needs to consist entirely of triangles labeled by $1$. Therefore, $\mathcal{T}$ is a Conway-Coxeter frieze pattern by Remark \ref{map remark}. Similarly, if $Q(v)=Q(\mathcal{T})=-3n-3$, then $\mathcal{T}$ is a twisted Conway-Coxeter frieze pattern.\\
        A row of a (twisted) Conway-Coxeter frieze pattern has just one arrow to the next row in $\Gamma_{2,n}(\mathbb{Z})$. Therefore, the strongly connected component of $v$ contains precisely the rows of the (twisted) Conway-Coxeter frieze pattern. In particular, $w$ can only be in the same strongly connected component as $v$ if it appears in the same (twisted) Conway-Coxeter frieze pattern.
    \end{proof}
\end{Lemma}

Now we turn our attention towards the rows with quiddity number different from $\pm(3n+3)$. 
For two arbitrary rows $v$ and $w$ with the same quiddity number, we want to construct a path from $v$ to $w$ in $\Gamma_{2,n}(\mathbb{Z})$. Our strategy for this is as follows: 
Let $\mathcal{F}_v$ and $\mathcal{F}_w$ be tame frieze patterns containing the rows $v$ and $w$ respectively.
Let $f_n$ be the surjective map from admissible labelings of an $(n+3)$-gon to tame frieze patterns over $\mathbb{Z}$ of height $n$. We want to construct a finite sequence of admissible labelings $L_1,\dots,L_k$ such that $f_n(L_1)=\mathcal{F}_v$ and $f_n(L_k)=\mathcal{F}_w$, and for each $i=1,\dots,k-1$ the tame frieze patterns $f_n(L_i)$ and $f_n(L_{i+1})$ have a common row $u_i$.
Since $u_i$ and $u_{i+1}$ are both part of the tame frieze pattern $f(L_{i+1})$, we get paths from $u_i$ to $u_{i+1}$ and from $v$ to $u_1$ and from $u_{k-1}$ to $w$. We then obtain a path from $v$ to $w$ by concatenating those paths.
\\
\\
The next two Lemmata will give the two ways consecutive admissible labelings $L_i$ and $L_{i+1}$ will be related in the sequence we want to construct.

\begin{Lemma}\label{reassigning}
    Let $L$ a be an admissible labeling of an $(n+3)$-gon with underlying triangulation $T$, and let $f_n$ be the surjective map assigning tame frieze patterns over $\mathbb{Z}$ to admissible labelings.
  
  Let $t_1$ and $t_2$ be adjacent triangles in the underlying triangulation of $L$ such that $t_1$ is labeled  $a$ and $t_2$ is labeled $-a$. Let $L'$ be a second admissible labeling of the same underlying triangulation such that the label of $t_1$ is $b$, the label of $t_2$ is $-b$, and all other labels are the same as in $L$. Then $f_n(L)$ and $f_n(L')$ have a common row.

    \begin{proof}
            We can assume $b\neq a$ because otherwise $L=L'$ and there is nothing to show.\\
            
            The edges of the  quadrilateral are either edges or diagonals in the underlying triangulation $T$ of $L$.
            We can obtain $L$ from the labeling of the  quadrilateral $\{t_1,t_2\}$ by attaching labeled triangulations to those sides of the  quadrilateral that need to become diagonals in $T$. If a side of the  quadrilateral is also a side in $T$, then we would not need to attach a labeled triangulation to this side. However, by assuming that we attach a labeled $2$-gon to those sides, we can avoid a case distinction. The $-1$-quiddity cycle of the labeled $2$-gon is $(0,0)$.

            With this trick in mind, $L$ can be obtained by gluing $4$ smaller labeled triangulations $L_1,L_2,L_3$ and $L_4$ to the sides of the  quadrilateral.
            We claim that the four labelings will be $\pm1$ admissible.
            Since $L$ is admissible, there is a partition $P$ of the triangles that are not labeled $\pm1$ in $L$ into polarized quadrilaterals. Similarly, since $L'$ is admissible, there is a partition $P'$ of the triangles not labeled $\pm1$ in $L'$ into polarized  quadrilaterals. We want to show that for at least one of these two partitions each of the polarized  quadrilaterals is contained in one of the $L_i$. \\
            Assume this is not the case for both $P$ and $P'$. 
            Then $t_1$ and $t_2$ have to form polarized quadrilaterals with some triangles contained in $L_i$. In particular, $t_1$ and $t_2$ are not labeled with $1$ or $-1$.
            We can split $L$ along the edge between $t_1$ and $t_2$. This gives us two $\pm1$ admissible labelings $L_+$ and $L_-$. 
            Similarly, we can split $L'$ into $\pm1$-admissible labelings $L_+'$ and $L_{-}'$.
            But now $L_+$ and $L_+'$ are two $\pm1$ admissible labelings such that all labels except the one of $t_1$ are the same.
            Since neither $a$ nor $b$ is $\pm1$ this is not possible. 
            \\\\

            Therefore, all of the $L_i$ are $\pm1$ admissible because we can group the triangles not labeled $1$ or $-1$ as in one of the partitions $P$ and $P'$.
            Thus, by adding the labels of all triangles adjacent to an edge, we obtain either $1$-quiddity cycles or $-1$-quiddity cycles. We can assume that $L_1$ and $L_2$ are attached to the triangle $t_1$, and $L_3$ and $L_4$ are attached to $t_2$. By gluing $L_1$ and $L_4$ together along the edges which connect them to the  quadrilateral, we obtain a labeling $L_{1,4}$.
            Similarly, $L_2$ and $L_3$ can be glued together along the edges connecting them to $t_2$. We call the resulting labeling $L_{2,3}$. We now compare the $\pm1$-quiddity cycles of all the different labelings. 
            If $(c^{(1)}_1,\dots,c^{(1)}_{r_1})$ is the $\pm1$-quiddity cycle obtained from $L_1$ and $(c^{(4)}_1,\dots,c^{(4)}_{r_4})$ is the $\pm1$-quiddity cycle obtained from $L_4$, then $L_{1,4}$ will yield the quiddity cycle 
            \begin{align*}
                (c^{(4)}_2,\dots,c^{(4)}_{r_4-1},c^{(4)}_{r_4}+c^{(1)}_1,c^{(1)}_2,\dots,c^{(1)}_{r_1-1},c^{(1)}_{r_1}+c^{(4)}_1).
            \end{align*}
            \\
                
            Similarly, if $(c^{(2)}_1,\dots,c^{(2)}_{r_2})$ is the $\pm1$-quiddity cycle obtained from $L_2$ and $(c^{(3)}_1,\dots,c^{(3)}_{r_3})$ is the $\pm1$-quiddity cycle obtained from $L_3$, then $L_{2,3}$ will yield the quiddity cycle 
            \begin{align*}
                (c^{(2)}_2,\dots,c^{(2)}_{r_2-1},c^{(2)}_{r_2}+c^{(3)}_1,c^{(3)}_2,\dots,c^{(3)}_{r_3-1},c^{(3)}_{r_3}+c^{(2)}_{1}).
            \end{align*}
            (See \cite[Section 4]{C19})\\ \\
            Finally, the quiddity cycle obtained from $L$ will read  
            \begin{align*}
                (c^{(2)}_2,\dots,c^{(2)}_{r_2-1},c^{(2)}_{r_2}+c^{(3)}_{1},c^{(3)}_{2},\dots,c^{(3)}_{r_3-1},c^{(3)}_{r_3}+c^{(4)}_{1}-a,c^{(4)}_{2},\dots,c^{(4)}_{r_4-1},c^{4}_{r_4}+c^{(1)}_1,c^{(1)}_2,\dots,c^{(1)}_{r_1-1},c^{(1)}_{r_1}+c^{(2)}_1+a)
            \end{align*}
            up to suitable rotation. We write $\alpha:=c^{(4)}_1-a-c^{(2)}_1$.
            After this substitution, the quiddity cycle of $L$ can be rewritten as 
            \begin{align*}
                (c^{(2)}_2,\dots,c^{(2)}_{r_2-1},c^{(2)}_{r_2}+c^{(3)}_{1},c^{(3)}_{2},\dots,c^{(3)}_{r_3-1},c^{(3)}_{r_3}+c^{(2)}_{1}+\alpha,c^{(4)}_{2},\dots,c^{(4)}_{r_4-1},c^{(4)}_{r_4}+c^{(1)}_1,c^{(1)}_2,\dots,c^{(1)}_{r_1-1},c^{(1)}_{r_1}+c^{(4)}_1-\alpha).
            \end{align*}

            Similarly, the quiddity cycle of $L'$ will be 
            \begin{align*}
                (c^{(2)}_2,\dots,c^{(2)}_{r_2-1},c^{(2)}_{r_2}+c^{(3)}_{1},c^{(3)}_{2},\dots,c^{(3)}_{r_3-1},c^{(3)}_{r_3}+c^{(4)}_{1}-b,c^{(4)}_{2},\dots,c^{(4)}_{r_4-1},c^{(4)}_{r_4}+c^{(1)}_1,c^{(1)}_2,\dots,c^{(1)}_{r_1-1},c^{(1)}_{r_1}+c^{(2)}_1+b)
            \end{align*}
            up to a suitable rotation. 
            After substituting $\beta:=c^{(4)}_1-b-c^{(2)}_1$, the quiddity cycle can be rewritten as 
            \begin{align*}
                (c^{(2)}_2,\dots,c^{(2)}_{r_2-1},c^{(2)}_{r_2}+c^{(3)}_{1},c^{(3)}_{2},\dots,c^{(3)}_{r_3-1},c^{(3)}_{r_3}+c^{(2)}_{1}+\beta,c^{(4)}_{2},\dots,c^{(4)}_{r_4-1},c^{(4)}_{r_4}+c^{(1)}_1,c^{(1)}_2,\dots,c^{(1)}_{r_1-1},c^{(1)}_{r_1}+c^{(4)}_1-\beta).
            \end{align*}
            Therefore, Proposition \ref{common row Lemma} shows that $f_n(L)$ and $f_n(L')$ have a common row.

    \end{proof}
\end{Lemma}

\begin{Remark}
    We actually have to assume that $L'$ is again an admissible labeling in Lemma \ref{reassigning}. This is not automatically true, as shown by the following pathological example.

    \begin{figure}[H]
        \centering
        \includegraphics[width=0.8\linewidth]{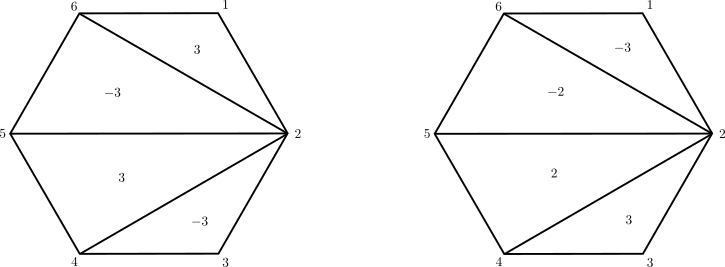}
        \caption{Reassigning a polarized quadrilateral might not preserve admissibility}
    \end{figure}
\end{Remark}
After relabeling the  quadrilateral in the middle, the labeling is not admissible anymore. This is because the chosen  quadrilateral is not part of the disjoint union of  quadrilaterals that makes the labeling admissible.

\begin{figure}[H]
    \centering
    \includegraphics[width=0.9\linewidth]{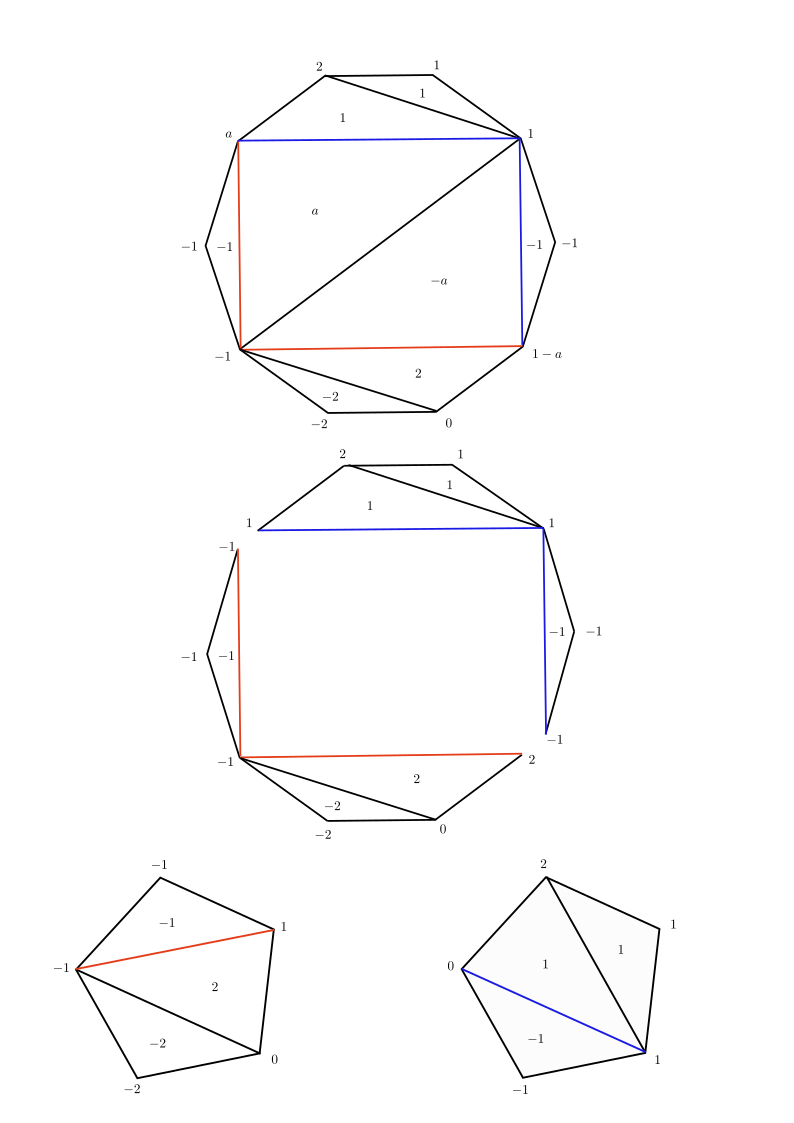}
    \caption{Illustration of the proof of Lemma \ref{reassigning}. Here we have $\alpha=a-1.$}
\end{figure}

\begin{Lemma} \label{flip of 0s}
      Let $L$ a be an admissible labeling of an $(n+3)$-gon, and let $f_n$ be the surjective map assigning tame frieze patterns over $\mathbb{Z}$ to admissible labelings.
      Let $t_1$ and $t_2$ be adjacent triangles both labeled by $0$. Then we get a new admissible labeling $L'$ by flipping the diagonal between $t_1$ and $t_2$ (This means in the  quadrilateral formed by $t_1$ and $t_2$ we replace the diagonal between $t_1$ and $t_2$ by the other diagonal of the  quadrilateral.) and labeling the two new triangles by $0$ again.\\ 
      Moreover, $f_n(L)=f_n(L')$.
    \begin{proof}
        Since no other polarized  quadrilaterals are affected, it is clear that $L'$ is again admissible.
        When computing a quiddity cycle from an admissible labeling, we sum up all the labels of triangles adjacent to a vertex.
        When passing from $L$ to $L'$, we get an additional summand $0$ in two of these sums and for two other sums a summand $0$ is removed. This does not change the resulting quiddity cycle.
        
    \end{proof}
\end{Lemma}

\begin{Lemma}\label{only 1s}
    Let $L$ be an admissible labeling of an $(n+3)$-gon. Then there is a sequence of admissible labelings
    \begin{align*}
        L=L_0,L_1,\dots,L_{r-1},L_r=L'
    \end{align*}
    such that for $i=0,\dots,r-1$ the tame frieze patterns $f_n(L_i)$ and $f_n(L_{i+1})$ have a common row, and all labels of $L'$ are either $1$ or $-1$. Moreover, the sum of the labels is the same for $L$ and $L'$.
    \begin{proof}
         Let $\{t_1,t_{-1}\},\dots,\{t_{r},t_{-r}\}$ be a partition of the triangles with a label different from $\pm1$ into polarized  quadrilaterals. We reassign the labels of each quadrilateral such that one triangle of each  quadrilateral is labeled $1$ the other is labeled $-1$ using Lemma \ref{reassigning}. Note that after each reassignment the labeling remains admissible because we only reassign polarized  quadrilaterals appearing in the partition. Since we have to reassign $r$ polarized  quadrilaterals, this gives us the admissible labelings $L_1$,\dots,$L_r$ such that $f_n(L_i)$ and $f_n(L_{i+1})$ have a common row for $i=0,\dots,r-1$.\\
         The sum of the labels will be the same for each $L_i$ because it is not changed by a reassignment.
    \end{proof}
\end{Lemma}

\begin{Lemma}\label{get triangulation right}
    Let $L$ and $L'$ be two admissible labelings of an $(n+3)$-gon such that all labels are either $1$ or $-1$ and both $L$ and $L'$ have at least one triangle labeled $1$ and one triangle labeled $-1$.
    Then there is a sequence 
    \begin{align*}
        L=L_0,L_1,\dots,L_{r-1},L_r=L''
    \end{align*}
     of admissible labelings such that for $i=0,\dots,r-1$ the frieze patterns $f(L_i)$ and $f(L_{i+1})$ have a common row, and the underlying triangulation of $L''$ is the same as that of $L'$. Furthermore, all labels of $L''$ are either $1$ or $-1$,  and the sum of all labels in $L$ and $L''$ will be the same.
    \begin{proof}
 
    It was shown in \cite{Law72} that any two triangulations of a convex polygon can be obtained from each other by applying a finite sequence of flips. We now want to show that we can always reassign labels using Lemma \ref{reassigning} such that a diagonal we want to flip is surrounded by two triangles labeled $0$.
    
    Assume that the first diagonal that needs to be flipped is adjacent to the triangles $t_0$ and $t_1$. 
    
    Since $1$s and $-1$s do not have to be part of a polarized  quadrilateral like other labels, we can reassign labels of  quadrilaterals without worrying about losing admissibility. By assumption, $L$ will have at least one label equal to $1$ and one label equal to $-1$. This means we can find a polarized  quadrilateral $\{t',t''\}$ with one triangle labeled $1$ and the other labeled $-1$ somewhere. Moreover, there is a sequence of overlapping  quadrilaterals
    \begin{align*}
        \{t',t''\}=\{t_s,t_{s-1}\},\{t_{s-1},t_{s-2}\},\dots,\{t_2,t_1\},\{t_{1},t_0\}
    \end{align*}
    Let $a_i$ be the label of $t_i$ in $L$ for $i=0,\dots,s$. Starting with $\{t_s,t_{s-1}\}$, we relabel the  quadrilaterals $\{t_i,t_{i-1}\}$ such that the new label of $t_{i-1}$ is $-a_{i-2}$. This way, the labels of the  quadrilateral $\{t_{i-1},t_{i-2}\}$ are $1$ and $-1$ and this  quadrilateral can be reassigned next. Finally, we are able to label the  quadrilateral $\{t_{1},t_0\}$ with two $0$s.\\
    This way, we obtain admissible labelings $L_{1},\dots,L_{s}$. By Lemma \ref{flip of 0s}, we can flip the diagonal between $t_{1}$ and $t_0$. This gives us another labeling $L_{s+1}$. \\
    Next, we reassign the two new triangles so that we have only labels $1$ and $-1$ again.
    This way we obtain one more labeling $L_{s+2}$.
    By Lemma \ref{reassigning} and Lemma \ref{flip of 0s}, the labelings $L_i$ and $L_{i+1}$ will always have a common row.\\
    
    If we proceed the same way for all other necessary flips, we end up with a sequence of admissible labelings as desired. \\
    Since neither reassigning polarized  quadrilaterals nor flipping an edge changes the sum of the labels, it will be the same for $L$ and $L''$.
    \end{proof}
\end{Lemma}

Now it remains to show that we can transform an admissible labeling consisting of $1$s and $-1$s into any other admissible labeling with the same underlying triangulation and the same sum of labels.
This will be done in the next Lemma.

\begin{Lemma}\label{get labels right}
    Let $\varepsilon \in\{\pm1\}$.
    Let $L$ and $L'$ be $\varepsilon$-admissible labelings of $(n+3)$-gons that have the same underlying triangulation and the same sum of labels. Furthermore, assume that each label of $L$ and each label of $L'$ is either $1$ or $-1$. Let $f_{n}$ be the map assigning tame frieze patterns to admissible labelings.
    Then there is a finite sequence of $\varepsilon$-admissible labelings \begin{align*}
        L=L_0,L_1,\dots,L_{r-1},L_r=L'
    \end{align*}
    such that the labels of $L_{i+1}$ is obtained from $L_i$ by relabeling one polarized quadrilateral for $i=0,\dots,r-1$. In particular, if $\varepsilon=1$, then $f_n(L_{i+1})$ and $f_n(L_i)$ have a common row for $i=0,\dots,r-1$.
    \begin{proof}
        We prove this by induction on the size of the  $\varepsilon$-admissible labeling.
        
        Let $q$ be the sum of all labels of $L$.
        If $n=0$, then $L=L'$ and there is nothing to do.  
        Assume $n \geq 1$ now.
        \\
        Again, if $L=L'$, there is nothing to do. Otherwise $-n-1< q <n+1$ and there is at least one triangle labeled $1$ and one triangle labeled $-1$. 
        Let $t_0$ be an ear of the underlying triangulation of $L$ (A triangle bounded by two sides of the $(n+3)$-gon and one diagonal). Let $a_0$ be the label of $t_0$ in $L'$.
        Let $t_1$ be the unique triangle adjacent to $t_0$. As in the proof of Lemma \ref{get triangulation right}, we can find a sequence of triangles $t_k,\dots,t_1,t_0$ such that $t_i$ and $t_{i-1}$ are adjacent for every $i=1,\dots,k$, and one of $t_k$ and $t_{k-1}$ is labeled $1$ and the other is labeled $-1$. We can now reassign the labels of the  quadrilaterals $\{t_{i+1},t_{i}\}$ starting with $\{t_k,t_{k-1}\}$ such that the new label of $t_{i}$ is the negative of the old label of $t_{i-1}$. Finally, we reassign the labels of $\{t_1,t_0\}$ such that the label of $t_0$ becomes $a_0$.
        These $k$ reassignments yield $\varepsilon$-admissible labelings $L=L_0,\dots,L_k$.\\
        We now remove the ear $t_0$ from $T$ and get a new triangulation $T'$. We write $L'\vert_{T'}$ and $L_k\vert_{T'}$ for the restrictions of $L'$ and $L_k$ to $T'$.
        They are $\varepsilon a_0$-admissible labelings. In both labelings the sum of all labels is $q-a_0$. By induction there is a sequence 
        \begin{align*}
             L_k\vert_{T'} = \tilde{L}_k,\tilde{L}_{k+1}\dots,\tilde{L}_{r-1}, \tilde{L}_r=   L'\vert_{T'} 
        \end{align*}
        of $\varepsilon a_0$admissible labelings such that $\tilde{L}_{i+1}$ is obtained from $\tilde{L}_i$ by reassigning a polarized quadrilateral for $i=k,\dots,r-1$.
        We can now extend each of those labelings to a labeling of $T$ by labeling $t_0$ with $a_0$. These will be $\varepsilon$-admissible labelings, and we obtain a sequence
         \begin{align*}
             L_k,L_{k+1}\dots,L_{r-1},L_r=L'
        \end{align*}
        where $L_{i+1}$ is obtained from $L_{i}$ by reassigning one polarized quadrilateral.
        Putting everything together, we obtain a sequence 
        \begin{align*}
            L=L_0,L_1,\dots,L_{r-1},L_r=L'
        \end{align*}
        where $L_{i+1}$ is obtained from $L_{i}$ by reassigning one polarized quadrilateral.\\
        If $\varepsilon=1$, Lemma \ref{reassigning} implies that $f_n(L_i)$ and $f_n(L_{i+1})$ have a common row for $i=0,\dots,r-1$.
  
    \end{proof}
\end{Lemma}
We are now ready for the proof of our main result.

\begin{proof}[Proof of Theorem \ref{connected component classification}]
    First, let $v$ and $w$ be in the same strongly connected component of $\Gamma_{2,n}(\mathbb{Z})$. Then by repeated application of Proposition \ref{common row Lemma}, we see that $Q(v)=Q(w)$. If $Q(v)=Q(w)=\pm(3n+3)$, we know that $v$ and $w$ are part of the same (twisted) Conway-Coxeter frieze pattern by Lemma \ref{extreme quiddity number}. \\
    \\
    Now assume $Q(v)=Q(w)$. If $v$ and $w$ are part of the same (twisted) Conway-Coxeter frieze pattern, then the rows of this frieze pattern give us a paths between $v$ and $w$ in both directions.
    \\
    Let $Q(v)=Q(w)\neq \pm(3n+3)$ now. Let $\mathcal{T}_v$ be a tame frieze pattern containing $v$, and let $\mathcal{T}_w$ a tame frieze pattern containing $w$. Let $L_v$ be an admissible labeling such that $f_n(L_v)=\mathcal{T}_v$, and let $L_w$ an admissible labeling such that $f_n(L_w)=\mathcal{T}_w$.
    By Lemma \ref{only 1s}, we find an admissible labeling $L_v'$ that only has labels $1$ and $-1$ and labelings
    \begin{align*}
        L_v=L_{v,0},L_{v,1},\dots,L_{v,r_v-1},L_{v,r_v}=L_v'
    \end{align*}
    with $f_n(L_{v,i})$ and $f_n(L_{v,i+1})$ having a common row for $i=0,\dots,r_v-1$.

    Also by applying Lemma \ref{only 1s} and reversing the resulting sequence, we get an admissible labeling $L_w'$
    that only has labels $1$ and $-1$ and a sequence of admissible labelings 
    \begin{align*}
        L_w'=L_{w,0},L_{w,1},\dots,L_{w,r_w-1},L_{w,r_w}=L_w
    \end{align*}
    with $f_n(L_{w,i})$ and $f_n(L_{w,i+1})$ having a common row for $i=0,\dots,r_w-1$.

    By first using Lemma \ref{get triangulation right} to get a labeling with the right triangulation and then Lemma \ref{get labels right} to get the right labels, we obtain a sequence of admissible labelings 
    \begin{align*}
        L_v'=L'_{0},L'_{1},\dots,L'_{r'-1},L'_{r'}=L_w'
    \end{align*}
    such that $f_n(L'_i)$ and $f_n(L'_{i+1})$ have a common row $u_i$ for $i=0,\dots,r'-1$.\\
    Note that the assumption $Q(v)=Q(w)$ implies that the labels of $L_v$ and $L_w$ sum up to the same value. This is also the case for all triangulations obtained from Lemma \ref{only 1s} and Lemma \ref{get triangulation right}. Therefore, it is actually possible to use Lemma \ref{get labels right} afterwards.   

    We can now combine these three sequences to obtain one longer sequence
    \begin{align}\label{long sequence}
        L_v=L_0,L_1,\dots,L_{r-1},L_r=L_w
    \end{align}
    of admissible labelings such that $f_n(L_i)$ and $f_n(L_{i+1})$ have a common row $u_i$ for $i=0,\dots,r-1$.
    We obtain paths from $u_{i}$ to $u_{i+1}$ for $i=0,\dots,r-1$ because both $u_{i}$ and $u_{i+1}$ are rows of $f_n(L_{i+1})$. Additionally $\mathcal{T}_v$ gives us a path from $v$ to $u_0$ and $\mathcal{T}_w$ gives us a path from $u_{r-1}$ to $w$. We now obtain a path from $v$ to $w$ by concatenating all those paths. \\
    By switching the roles of $v$ and $w$ in the argument above, we also obtain a path from $w$ to $v$.
\end{proof} 

\begin{Remark}
    The weakly connected components of $\Gamma_{2,n}(\mathbb{Z})$ are the same as the strongly connected components because there are no arrows between different strongly connected components. 
    Since an arrow between two vertices $v$ and $w$ already implies $Q(v)=Q(w)$, there could only be arrows between the vertices of two strongly connected components corresponding to (twisted) Conway-Coxeter frieze patterns. But this can not happen because each (twisted) Conway-Coxeter frieze pattern is uniquely determined by one of its rows.\\
    Alternatively, one could use the periodicity of tame frieze patterns to see this without using Theorem \ref{connected component classification}.
\end{Remark}

\begin{Example}
    Since frieze patterns over $\mathbb{Z}$ correspond to infinite paths in $\Gamma_{2,n}(\mathbb{Z})$, we can use Theorem \ref{connected component classification} to decide whether two rows of integers can be part of the same frieze pattern over $\mathbb{Z}$. For example, consider the tuples $(0,1,-4,-1,-1,4,1,0)$ and $(0,1,0,-1,2,3,1,0)$. The tuple $(0,1,-4,-1,-1,4,1,0)$ appears as a row in the tame frieze pattern 
    \begin{align*}
    \begin{NiceMatrixBlock}[auto-columns-width]
    \begin{NiceMatrix}
    \ddots&&&&&\ddots\\
        \color{red} \color{red}0& \color{red}1& \color{red}-4& \color{red}-1& \color{red}-1& \color{red}4& \color{red}1& \color{red}0\\
        &0&1&0&-1&3&1&1&0\\
        &&0&1&5&-16&-5&-4&1&0\\
        &&&0&1&-3&-1&-1&0&1&0\\
        &&&&0&1&0&-1&-1&5&1&0\\
        &&&&&0&1&4&3&-16&-3&1&0\\
        &&&&&&0&1&1&-5&-1&0&1&0\\
        &&&&&&&&\ddots&&&&&\ddots\\    
    \end{NiceMatrix}
    \end{NiceMatrixBlock}.
    \end{align*}
     This shows that $Q((0,1,-4,-1,-1,4,1,0))=3.$\\
    
The tuple $(0,1,0,-1,2,3,1,0)$ appears for example as a row in the tame frieze pattern
     \begin{align*}
    \begin{NiceMatrixBlock}[auto-columns-width]
    \begin{NiceMatrix}
    \ddots&&&&&\ddots\\
        \color{red} 0&\color{red} 1&\color{red} 0&\color{red} -1&\color{red} 2&\color{red} 3&\color{red} 1&\color{red} 0\\
        &0&1&1&-3&-4&-1&1&0\\
        &&0&1&-2&-3&-1&0&1&0\\
        &&&0&1&1&0&-1&1&1&0\\
        &&&&0&1&1&2&-3&-2&1&0\\
        &&&&&0&1&3&-4&-3&1&1&0\\
        &&&&&&0&1&-1&-1&0&1&1&0\\
        &&&&&&&&\ddots&&&&&\ddots\\        
    \end{NiceMatrix}
    \end{NiceMatrixBlock}.
    \end{align*}
   Therefore, we also have $Q((0,1,0,-1,2,3,1,0))=3$. By Theorem \ref{connected component classification}, there should be a frieze pattern containing both these rows. Indeed, they are both contained in the periodic wild frieze pattern 
      \begin{align*}
    \begin{NiceMatrixBlock}[auto-columns-width]
    \begin{NiceMatrix}
    \ddots&&&&&\ddots\\
        \color{red} \color{red}0& \color{red}1& \color{red}-4& \color{red}-1& \color{red}-1& \color{red}4& \color{red}1& \color{red}0\\
        &0&1&0&-1&3&1&1&0\\
        &&0&1&0&-1&0&1&1&0\\
        &&&\color{red} 0&\color{red} 1&\color{red} 0&\color{red} -1&\color{red} 2&\color{red} 3&\color{red} 1&\color{red} 0\\
        &&&&0&1&0&-1&-1&0&1&0\\
        &&&&&0&1&1&0&-1&0&1&0\\
        &&&&&&0&1&1&-5&-1&0&1&0\\
        &&&&&&&\color{red} \color{red}0& \color{red}1& \color{red}-4& \color{red}-1& \color{red}-1& \color{red}4& \color{red}1& \color{red}0\\
        &&&&&&&&&\ddots&&&&&\ddots\\        
    \end{NiceMatrix}
    \end{NiceMatrixBlock}.
    \end{align*}
  
\end{Example}

We now want to count the connected components of $\Gamma_{2,n}(\mathbb{Z})$. Using Theorem \ref{connected component classification}, we only need to count (twisted) Conway-Coxeter frieze patterns up to the action of a cyclic group of order $n+3$ and determine the number of possible values of $Q(v)$.

\begin{Definition}
    For $m \in \mathbb{Z}_{\geq 1}$ let $C_m:=\frac{1}{m+1}\binom{2m}{2}$ denote the $m$-th Catalan number. For $q \in \mathbb{Q} \setminus \mathbb{Z}_{\geq 1}$ we define $C_q:=0.$
\end{Definition}
We will use the following Lemma from Bowman and Regev.
\begin{Lemma}\cite[Theorem 29]{BR14} \label{triangulations up to rotation}\footnote{\phantom{r}Triangulations up to rotational symmetry were first enumerated in \cite[Equation (6.4)]{Bro64} but not in terms of Catalan numbers. See also \cite{Sloane} for various interpretations of the sequence.}
    The number of triangulations of a convex $m$-gon up to rotational symmetry is given by 
    \begin{equation*}
        \frac{1}{m}C_{m-2}+\frac{1}{2}C_{\frac{m}{2}-1}+\frac{2}{3}C_{\frac{m}{3}-1}.
    \end{equation*}
\end{Lemma}

\begin{Corollary}\label{CC-frieze patterns up to rotations}
    Let $n\in \mathbb{Z}_{\geq 3}$. Then the number of Conway-Coxeter frieze patterns of height $n$ up to rotation of rows is 
       \begin{equation*}
        \frac{1}{n+3}C_{n+1}+\frac{1}{2}C_{\frac{n+1}{2}}+\frac{2}{3}C_{\frac{n}{3}}.
    \end{equation*}
    \begin{proof}
        The bijection between triangulations of convex $(n+3)$-gons and Conway-Coxeter frieze patterns is equivariant with respect to the actions of the cyclic group of order $n+3$ on triangulated $(n+3)$-gons and tame frieze patterns of height $n$. 
        In particular, these group actions have the same number of orbits. Now the claim follows from Lemma \ref{triangulations up to rotation}.
    \end{proof}
\end{Corollary}

\begin{Lemma}
    Let $n\in \mathbb{Z}_{\geq 0}$. Then there are $\left\lceil \frac{n+2}{2}\right\rceil$ different possible values for the quiddity number of rows of frieze patterns over $\mathbb{Z}$ of height $n$.
    \begin{proof}
        Let $\mathcal{F}_v$ be a tame frieze pattern containing the row $v$.
        If $L$ is an admissible labeling with $f_n(L)=\mathcal{F}_v$, then $Q(\mathcal{F}_v)$ is three times the sum of the labels of $L$.
        This means we can count possible sums of labels in admissible labelings instead. Moreover, since any label $a\notin \{\pm1\}$ appears in a  quadrilateral with a label $-a$, we already get all possible sums of labels if we only allow the labels $1$ and $-1$. 
        Let $l_{-1}$ be the number of triangles labeled $-1$ in some admissible labeling $T$. Then the sum of all labels is $n+1-2l_{-1}$.  
        Now there are $n+1$ triangles and the sign of the labeling is $1$ if and only if there is an even number of labels $-1$.
        Therefore, the possible values for $l_{-1}$ are $0,2,\dots,n-2,n$ if $n$ is even and $0,2,\dots,n-1,n+1$ if $n$ is odd. In both cases there are $\left\lceil\frac{n+2}{2}\right\rceil$ possible values for $l_{-1}$ and $n+1-2l_{-1}$.
    \end{proof}
\end{Lemma}

\begin{Proposition}\label{counting connected components}
    Let $n \in \mathbb{Z}_{\geq3}$. The number of strongly connected components of $\Gamma_{2,n}(\mathbb{Z})$ is 
    \begin{equation*}
        \frac{1}{n+3}C_{n+1}+\frac{2}{3}C_{\frac{n}{3}}+\frac{n}{2}
    \end{equation*}
    if $n$ is even and 
    \begin{equation*}
        \frac{2}{n+3}C_{n+1}+C_{\frac{n+1}{2}}+\frac{4}{3}C_{\frac{n}{3}}+\frac{n-1}{2}
    \end{equation*}
    if $n$ is odd.
    \begin{proof}
        If $n$ is even, then by Theorem \ref{connected component classification} and Corollary \ref{CC-frieze patterns up to rotations} there are $\frac{1}{n+3}C_{n+1}+\frac{2}{3}C_{\frac{n}{3}}$ strongly connected components where the quiddity number of the vertices is $3n+3$, but there are no strongly connected components where the quiddity number of the vertices is $-(3n+3)$. \\
        (Since $n$ is even, we can omit the term $\frac{1}{2}C_{\frac{n+1}{2}}$ appearing in Lemma \ref{CC-frieze patterns up to rotations}.)  
        \\
        If $n$ is odd, there are also $\frac{1}{n+3}C_{n+1}+\frac{1}{2}C_{\frac{n+1}{2}}+\frac{2}{3}C_{\frac{n}{3}}$ strongly connected components where the quiddity number of the vertices is $-(3n+3)$. They correspond to twisted Conway-Coxeter frieze patterns up to rotations. 
        \\ 
        There are  $\frac{n}{2}$ additional strongly connected components corresponding to quiddity numbers different from $\pm(3n+3)$ if $n$ is even and $\frac{n-1}{2}$ if $n$ is odd.

    \end{proof}
\end{Proposition}

\begin{Example}
    The following table gives the number of strongly connected components of $\Gamma_{2,n}(\mathbb{Z})$ for $n\leq10$.
    \begin{center}
\begin{tabular}{ c| c c c c c c c c c c c }
 $n$ & $0$ & $1$ & $2$ & $3$ & $4$ & $5$ & $6$ & $7$ & $8$ & $9$ & $10$ \\ 
 \hline
   \#components with $Q(v)=3n+3$  & $1$& $1$ & $1$ & $4$ & $6$ & $19$& $49$& $150$ & $442$ & $1424$ & $4522$\\
   \#components with $Q(v)=-(3n+3)$  & $0$& $1$ & $0$ & $4$ & $0$ & $19$& $0$& $150$ & $0$ & $1424$ & $0$\\
   \#components with $Q(v)\neq \pm(3n+3)$  & $0$& $0$ & $1$ & $1$ & $2$ & $2$& $3$& $3$ & $4$ & $4$ & $5$\\
   \#all components & $1$& $2$ & $2$ & $9$ & $8$ & $40$& $52$& $303$ & $446$ & $2852$ & $4527$\\
\end{tabular}
\end{center}
\end{Example}

\begin{Remark}
    By taking a close look at the steps used to prove Theorem \ref{connected component classification}, we can see that each strongly connected component of $\Gamma_{2,n}(\mathbb{Z})$ has finite diameter. 
    One can also derive upper bounds depending only on $n$ this way. We first bound the lengths of the sequences obtained from Lemma \ref{only 1s}, Lemma \ref{get triangulation right} and Lemma \ref{get labels right}. 
    Since a triangulated $(n+3)$-gon consists of $n+1$ triangles, we need to reassign at most $\frac{n+1}{2}$ polarized  quadrilaterals for Lemma \ref{only 1s}. In Lemma \ref{get triangulation right}, we need to do at most $2n$ flips.
    (In \cite[Theorem 2.1]{wood1982note}, this bound was proved for the rotation distance of labeled binary trees. In \cite{sleator1986rotation}, this was reformulated in terms of triangulations and flips and improved for sufficiently large $n$.) 
    
    For each flip we have to reassign at most $n$ overlapping polarized  quadrilaterals, then do a flip and finally do one more reassignment. So in total we need at most $2n(n+2)$ new labelings for Lemma \ref{get triangulation right}. Finally, for Lemma \ref{get labels right} we need at most $\sum_{k=1}^{n}k=\frac{n(n+1)}{2}$ additional labelings. Since we use Lemma \ref{only 1s} twice, we see that  
    \begin{align*}
        r\leq2\cdot\frac{n+1}{2}+2n(n+2)+\frac{(n+1)n}{2}=\frac{5}{2}n^{2}+\frac{11}{2}n+1
    \end{align*}
    in (\ref{long sequence}). The directed paths between the $u_i$ and $u_{i+1}$ have length at most $n+2$ because of the periodicity of tame frieze patterns and we concatenate $r+1$ paths.
    Therefore, the diameter of each infinite strongly connected component of $\Gamma_{2,n}(\mathbb{Z})$ is bounded by
    \begin{align*}
        (r+1)(n+2)\leq \frac{5}{2}n^{3}+\frac{21}{2}n^{2}+13n+4.
    \end{align*}
    The finite strongly connected components have diameter at most $n+2$.
    \\
    \\
    It should be possible to obtain a significantly better bound for the diameter of the infinite components.
\end{Remark}
\label{section_components}
\section{\texorpdfstring{More properties of $\boldsymbol{\Gamma_{2,n}(\mathbb{Z})}$}{More properties}}
In this section we want to collect some more facts about the graph $\Gamma_{2,n}(\mathbb{Z})$. First, we characterize $2$-periodic frieze patterns, which correspond to loops or $2$ cycles in $\Gamma_{2,n}(\mathbb{Z})$.
The $2$-periodic frieze patterns of types $A_n,D_n,E_6,E_7,E_8$ over $\mathbb{R}_{>0}$ were classified in \cite{king20252}. Frieze patterns of period $1$ over $\mathbb{R}_{>0}$ were already considered in \cite{CC73a}. This was later generalized in \cite{germain2025fixed}.

Next, we show that each finite simple directed graph appears as an induced subgraph of $\Gamma_{2,n}(\mathbb{Z})$ for sufficiently large $n$.

\begin{Lemma}\label{every second diagonal constant}
    Let 
    \begin{align*}
          \begin{NiceMatrix}
            \ddots&&&&\ddots\\
            0&1&x_1&\dots&x_n&1&0\\
            &0&1&y_1&\dots&y_n&1&0\\
            &&0&1&x_1&\dots&x_n&1&0\\
            &&&0&1&y_1&\dots&y_n&1&0\\
            &&&&&\ddots&&&&\ddots\\
        \end{NiceMatrix}
    \end{align*}
    be a $2$-periodic frieze pattern over $\mathbb{Z}$.
    Let $r$ be the smallest integer such that either $x_r=0$ or $y_r=0$ .
    Then $x_{2k}=y_{2k}$ for $k=1,\dots,\operatorname{min}\left(\left\lfloor\frac{n}{2}\right\rfloor,\left\lfloor\frac{r-1}{2}\right\rfloor\right)$.
    Moreover, if $x_{2k}x_{2k-2}\neq-1$,
    we have 
    \begin{equation}\label{even diagonal recursion}
          x_{2k+2}=\frac{x_{2k}^{3}-2x_{2k}-x_{2k-2}}{x_{2k}x_{2k-2}+1}.    
    \end{equation} for $k=1,\dots,\operatorname{min}\left(\left\lfloor\frac{n}{2}\right\rfloor,\left\lfloor\frac{r-1}{2}\right\rfloor\right)$.
\begin{proof}
    For the first claim note that the $SL_2$ rule tells us that
    \begin{equation*}
          x_{2k}y_{2k-2}=x_{2k-1}y_{2k-1}-1=y_{2k}x_{2k-2}.
    \end{equation*}
Since $x_0=y_0=1$, the equality $x_{2k}=y_{2k}$ now follows by induction using $x_{2k-2}=y_{2k-2}\neq 0$.
For proving \eqref{even diagonal recursion}, we will use the following special cases of the $SL_2$ rule:
\begin{equation}\label{rec1}
    x_{2k+1}y_{2k-1}=x^{2}_{2k}-1=y_{2k+1}x_{2k-1}.
\end{equation}
\begin{equation}\label{rec2}
    x_{2k}x_{2k-2}+1=x_{2k-1}y_{2k-1}.
\end{equation}

From those equations, we conclude
\begin{align*}
    x^{4}_{2k}-2x^{2}_{2k}+1&\phantom{.}=\left(x^{2}_{2k}-1\right)^{2}\\ 
    &\overset{\eqref{rec1}}{=} x_{2k+1}y_{2k-1}y_{2k+1}x_{2k-1}\\
    &\overset{\eqref{rec2}}{=}(x_{2k}x_{2k+2}+1)(x_{2k-2}x_{2k}+1)\\
    &\overset{\phantom{\eqref{rec2}}}{=}x_{2k-2}x^{2}_{2k}x_{2k+2}+x_{2k}x_{2k+2}+x_{2k-2}x_{2k}+1.
\end{align*}
Now if $x_{2k}x_{2k-2}\neq -1$, we can solve for $x_{2k+2}$ and obtain \eqref{even diagonal recursion}.
\end{proof}
\end{Lemma}

The following Lemma will later be used for restricting the possible values of $x_2=y_2$ in $2$-periodic frieze patterns.
\begin{Lemma}\label{sequence growth}
Let $x_{2k}$ be recursively defined by equation \eqref{even diagonal recursion}, $x_0:=1$ and specifying $x_2$ to some integer not contained in $\{-1,0,1,2\}$. Then we have $\lvert x_{2k+2}\rvert>\lvert x_{2k}\rvert$ for all $k\in \mathbb{Z}_{\geq 0}$.
\begin{proof}
    We first consider $x_2 \geq 3$. We want to prove inductively that $x_{2k}\geq x_{2k-2}+2$ implies $x_{2k+2}\geq x_{2k}+2$. 
    \begin{align*}
        x_{2k+2}&=\frac{x_{2k}^{3}-2x_{2k}-x_{2k-2}}{x_{2k}x_{2k-2}+1}  \\
            &\geq \frac{x_{2k}^{3}-3x_{2k}+2}{x_{2k}(x_{2k}-2)+1}\\
            &=\frac{(x_{2k}-1)^{2}(x_{2k}+2)}{(x_{2k}-1)^{2}}\\
            &=x_{2k}+2.
    \end{align*}
    Here we used the assumption $x_{2k} \geq x_{2k-2}$ both in the numerator and the denominator.
    \\
     
    Now we assume $x_2\leq -2$. We want to show inductively that $x_{2k+2}\leq-2x_{2k}\leq 0$ if $k$ is even and $x_{2k+2} \geq -2x_{2k} \geq 0$ if $k$ is odd.\\
    By assumption $x_2 \leq-2 \leq -2 x_0 \leq 0$ holds.
    Assume first that $k$ is odd. Then by induction $x_{2k}\leq -2x_{2k-2}$ and $x_{2k-2} \geq 0$ and we can estimate
    \begin{align*}
    x_{2k+2}&=\frac{x_{2k}^{3}-2x_{2k}-x_{2k-2}}{x_{2k}x_{2k-2}+1}\\
            &\geq\frac{x_{2k}^{3}-2x_{2k}}{x_{2k}x_{2k-2}+1}\\
            &\geq \frac{-2x_{2k-2}x_{2k}^{2}-2x_{2k}}{x_{2k}x_{2k-2}+1}\\
            &=-2x_{2k}.
    \end{align*}
    Now if $k$ is even, we have $x_{2k} \geq -2x_{2k-2}$ and $x_{2k-2} \leq 0$ and we can estimate 
     \begin{align*}
    x_{2k+2}&=\frac{x_{2k}^{3}-2x_{2k}-x_{2k-2}}{x_{2k}x_{2k-2}+1}\\
            &\leq\frac{x_{2k}^{3}-2x_{2k}}{x_{2k}x_{2k-2}+1}\\
            &\leq \frac{-2x_{2k-2}x_{2k}^{2}-2x_{2k}}{x_{2k}x_{2k-2}+1}\\
            &=-2x_{2k}.
    \end{align*}
    This tells us that $\lvert x_{2k+2}\rvert \geq 2\lvert x_{2k}\rvert >\lvert x_{2k}\rvert$.
\end{proof}
\end{Lemma}

We also get an analogous statement if we set $x_0:=-1$.

\begin{Corollary}\label{-1 version}
    Let $x_{2k}$ be recursively defined by equation \eqref{even diagonal recursion}, $x_0:=-1$ and specifying $x_2$ to some integer not in $\{-1,-2,0,1\}$. Then we have $\lvert x_{2k+2}\rvert>\lvert x_{2k}\rvert$ for all $k\in \mathbb{Z}_{\geq 0}$.
    \begin{proof}
        If we multiply $x_0$ and $x_2$ by $-1$, then all other $x_{2k}$ will also be multiplied by $-1$ because the right side of \eqref{even diagonal recursion} is an odd symmetric function in $x_{2k}$ and $x_{2k-2}$. This way, the claim follows directly from Lemma \ref{sequence growth}.
    \end{proof}
\end{Corollary}
\newpage

\begin{Lemma}
    Let $\mathcal{F}$ be a $2$-periodic frieze pattern over $\mathbb{Z}$. Then we can obtain $\mathcal{F}$ from the segments
    \begin{NiceMatrixBlock}[auto-columns-width]
       \begin{align}
        \begin{NiceMatrix}\label{height0 segments}
            \color{red}\ddots&\color{red}\ddots\\
            &\color{red}1&\color{red}\pm1\\
            &&\color{red}1&\color{red}\pm1\\
            &&&\color{red}1&\color{red}\pm1\\
            &&&&\color{red}\ddots&\color{red}\ddots\\
        \end{NiceMatrix}
         \begin{NiceMatrix}
            \color{red}\ddots&\color{red}\ddots\\
            &\color{red}-1&\color{red}\pm1\\
            &&\color{red}-1&\color{red}\pm1\\
            &&&\color{red}-1&\color{red}\pm1\\
            &&&&\color{red}\ddots&\color{red}\ddots\\
        \end{NiceMatrix}
    \end{align}

     \begin{align}\label{height1 segments}
        \begin{NiceMatrix}
            \color{red}\ddots&&\color{red}\ddots\\
            &\color{red}1&\pm1&\color{red}{1}\\
            &&\color{red}1&\pm2&\color{red}{1}\\
            &&&\color{red}1&\pm1&\color{red}{1}\\
            &&&&\color{red}\ddots&&\color{red}\ddots\\
        \end{NiceMatrix}
        \begin{NiceMatrix}
            \color{red}\ddots&&\color{red}\ddots\\
            &\color{red}-1&\pm1&\color{red}{-1}\\
            &&\color{red}-1&\pm2&\color{red}{-1}\\
            &&&\color{red}-1&\pm1&\color{red}{-1}\\
            &&&&\color{red}\ddots&&\color{red}\ddots\\
        \end{NiceMatrix}
    \end{align}
\small{
    \begin{align}\label{height3 segments}
        \begin{NiceMatrix}
            \color{red}\ddots&&&&\color{red}\ddots&&&\\
            &\color{red}1&\pm1&2&\pm1&\color{red}1\\
            &&\color{red}1&\pm3&2&\pm3&\color{red}1\\
            &&&\color{red}1&\pm1&2&\pm1&\color{red}1\\
            &&&&\color{red}\ddots&&\color{red}\ddots\\
        \end{NiceMatrix}
\begin{NiceMatrix}
            \color{red}\ddots&&&&\color{red}\ddots&&&\\
            &\color{red}-1&\pm1&2&\pm1&\color{red}-1\\
            &&\color{red}-1&\pm3&2&\pm3&\color{red}-1\\
            &&&\color{red}-1&\pm1&2&\pm1&\color{red}-1\\
            &&&&\color{red}\ddots&&\color{red}\ddots\\
        \end{NiceMatrix}
    \end{align}
    }
     \begin{align}\label{0 segment}
        \begin{NiceMatrix}
            \color{red}\ddots&&\color{red}\ddots\\
            &\color{red}\pm1&a&\color{red}\mp1\\
            &&\color{red}\pm1&0&\color{red}\mp1&\\
            &&&\color{red}\pm1&a&\color{red}\mp1&\\
            &&&&\color{red}\ddots&&\color{red}\ddots
        \end{NiceMatrix}
        \end{align}
        \begin{center}
            for some $a \in \mathbb{Z}$.
        \end{center}
         \begin{align}\label{start-end segment}
        \begin{NiceMatrix}
            \ddots&\color{red}\ddots\\
            &0&\color{red}1\\
            &&0&\color{red}1\\
            &&&0&\color{red}1\\
            &&&&\ddots&\color{red}\ddots\\
        \end{NiceMatrix}
        \begin{NiceMatrix}\
            \color{red}\ddots&\ddots\\
            &\color{red}1&0\\
            &&\color{red}1&0\\
            &&&\color{red}1&0\\
            &&&&\color{red}\ddots&\ddots\\
        \end{NiceMatrix}
    \end{align}
\end{NiceMatrixBlock}
    and their rotations by gluing them together along the red diagonals.
    \begin{proof}
        If $\mathcal{F}$ has some entry equal to $0$, then around the $0$ the frieze pattern looks like \begin{align*}
        \begin{NiceMatrix}
            &\pm1&\\
            \mp1&0&\pm1\\
            &\mp1&&
        \end{NiceMatrix}.
    \end{align*}
    The $2$-periodicity now tells us that the diagonal containing the $0$, together with the previous and next diagonal, form a segment of type \eqref{0 segment}. 
    If $\mathcal{F}$ has $r$ diagonals containing only entries $1$ and $-1$, we can split $\mathcal{F}$ into $r+1$ segments. We obtain segments of type \eqref{start-end segment} from the bounding diagonals of $\mathcal{F}$.
    Any segment that contains a $0$ is of type \eqref{0 segment} as seen above. Otherwise, the segment is of one of the forms
    \begin{align}\label{nonzero-segment}
    \begin{NiceMatrix}
            \color{red}\ddots&&&&\color{red}\ddots\\
            &\color{red}1&x_1&\dots&x_k&\color{red}{\pm1}\\
            &&\color{red}1&y_1&\dots&y_k&\color{red}{\pm1}\\
            &&&\color{red}1&x_1&\dots&x_k&\color{red}{\pm1}\\
            &&&&\color{red}\ddots&&&&\color{red}\ddots\\
        \end{NiceMatrix}
        \begin{NiceMatrix}
            \color{red}\ddots&&&&\color{red}\ddots\\
            &\color{red}-1&x_1&\dots&x_k&\color{red}{\pm1}\\
            &&\color{red}-1&y_1&\dots&y_k&\color{red}{\pm1}\\
            &&&\color{red}-1&x_1&\dots&x_k&\color{red}{\pm1}\\
            &&&&\color{red}\ddots&&&&\color{red}\ddots\\
        \end{NiceMatrix}
    \end{align}
     with $x_i\neq 0$ and $y_i\neq 0$ for $i=1,\dots,k$.\\
     \\
    Now it suffices to show that all segments of type \eqref{nonzero-segment} appear in the list above.
    Assume first that the segment starts with a diagonal of $1$s. 
    By Lemma \ref{every second diagonal constant}, every second diagonal is constant.
    Moreover, the only possible values for $x_2=y_2$ are $-1,1,0,2$ by Lemma \ref{sequence growth} because otherwise we would never get a diagonal of $1$s or $-1$s again.\footnote{We would still get an infinite frieze pattern as defined in \cite[Definition 1.1]{tschabold2015arithmetic}.} 
    Since $x_1y_1=x_2+1$, we can additionally rule out $x_2=-1$ because this would force one of $x_1$ and $y_1$ to be $0$.
    From the same equation, we see that $(x_1,y_1)\in\{(\pm1,\pm1),(\pm1,\pm2),(\pm2,\pm1),(\pm1,\pm3),(\pm3,\pm1)\}$.
    From the values of $x_1$ and $x_2$, we can recover the rest of the segment. This way, we see that the segment is one of those appearing in the list above.
    \\
    If the segment starts with a diagonal of $-1$s, we proceed similarly. Using Corollary \ref{-1 version}, we conclude that the possible values for $x_2=y_2$ are $-1,-2,0,1$. From the equation $x_1y_1=-x_2+1$, we can rule out $x_2=1$. We see that again $(x_1,y_1)\in\{(\pm1,\pm1),(\pm1,\pm2),(\pm2,\pm1),(\pm1,\pm3),(\pm3,\pm1)\}$. By computing the next rows from these values, we see that we again only obtain segments from the list above.
    \end{proof}
    
\end{Lemma}

\begin{Remark}
    A $2$-periodic frieze pattern can correspond to a loop or a $2$-cycle in $\Gamma_{2,n}(\mathbb{Z})$.
    We get a loop if each row of the frieze pattern is the same. This means loops correspond to frieze patterns glued together from segments of type \eqref{height0 segments}, type \eqref{start-end segment} and segments of type \eqref{0 segment} with $a=0$. If some other segment is used, the frieze pattern corresponds to a $2$-cycle.\\
\end{Remark}

\begin{Proposition}\label{embed finite digraphs}
    Let $G=(V,A)$ be a finite simple directed graph. Then there is $n\in \mathbb{Z}_{\geq 0}$ such that $G$ is isomorphic to an induced subgraph of $\Gamma_{2,n}(\mathbb{Z})$.
    \begin{proof}
        We can assume that $V=\{1,\dots,r\}$ for some $r$ and $A\subseteq V^{2}$.
        We will first construct vertices $V^{0}_1,\dots,V^{0}_r$ such that their induced subgraph is the complete directed graph on $r$ vertices. In the second step we will modify the vertices to get rid of arrows that do not exist in $G$. We will write $x||y$ for the concatenation of two tuples $x$ and $y$.

         Let $R_0:=
        (1,0,-1,0)$ and $R_1:=
        (1,1,-1,0)$. Then for $1\leq k \leq r$ consider the tuple 
    \begin{align*}
        V^{0}_k:=(0)||\underset{k-1 \text{ times}}{\underbrace{R_0||\dots||R_0}}||R_1||\underset{r-k \text{ times}}{\underbrace{R_0||\dots||R_0}}||
        (\begin{NiceMatrix}
        1&0  \\
    \end{NiceMatrix}).
    \end{align*}
    For all $l\neq k$ we have $V^{0}_k\neq V^{0}_l$ and both arrows between $V^{0}_k$ and $V^{0}_l$ exist in $\Gamma_{2,n}(\mathbb{Z})$. 
    Moreover, the $R_1$ in the middle ensures that $V^{0}_k$ has no arrow to itself.
    Therefore, the induced subgraph of $V^{0}_1,\dots,V^{0}_r$ will be a complete directed graph of size $r$. Now we need to get rid of superfluous arrows.\\
    We define
    $A:=
    (-1,1,1,0,-1,0,1,0 )$, $B:=
    (-1,0,1,1,-1,0,1,0)$ and $C:=(-1,0,1,0,-1,0,1,0)$.
    The six segments

    \begin{align*}
    \begin{NiceMatrixBlock}[auto-columns-width]
    \begin{NiceMatrix}
        A:&-1&1&1&0&-1&0&1&0  \\
        B:&&-1&0&1&1&-1&0&1&0
    \end{NiceMatrix}
     \end{NiceMatrixBlock}
    \\\\
    \begin{NiceMatrixBlock}[auto-columns-width]
    \begin{NiceMatrix}
        A:&-1&1&1&0&-1&0&1&0  \\
        C:&&-1&0&1&0&-1&0&1&0
    \end{NiceMatrix}
    \end{NiceMatrixBlock}
    \\\\
    \begin{NiceMatrixBlock}[auto-columns-width]
     \begin{NiceMatrix}
        C:& -1&0&1&0&-1&0&1&0\\
        A:&&-1&1&1&0&-1&0&1&0 
    \end{NiceMatrix}
    \end{NiceMatrixBlock}
    \\\\
    \begin{NiceMatrixBlock}[auto-columns-width]
    \begin{NiceMatrix}
        B:&-1&0&1&1&-1&0&1&0\\
        C:&& -1&0&1&0&-1&0&1&0
    \end{NiceMatrix}
    \end{NiceMatrixBlock}
    \\\\
    \begin{NiceMatrixBlock}[auto-columns-width]
    \begin{NiceMatrix}
        C:&-1&0&1&0&-1&0&1&0\\
        B:&&-1&0&1&1&-1&0&1&0
    \end{NiceMatrix}
    \end{NiceMatrixBlock}
     \\\\
    \begin{NiceMatrixBlock}[auto-columns-width]
    \begin{NiceMatrix}
        C:&-1&0&1&0&-1&0&1&0\\
        C:&& -1&0&1&0&-1&0&1&0
    \end{NiceMatrix}
    \end{NiceMatrixBlock}
    \end{align*}
    
    satisfy the $SL_2$ rule everywhere, but the segment
    \begin{align*}
    \begin{NiceMatrixBlock}[auto-columns-width]
         \begin{NiceMatrix}
         B:&-1&0&1&1&-1&0&1&0\\
        A:&&-1&1&1&0&-1&0&1&0  
    \end{NiceMatrix}
    \end{NiceMatrixBlock}
    \end{align*}
    has one $2\times2$ block with determinant $0$. 
    Let $(i,j)$ be an arrow that does not appear in $G$.
    We can define 
    \begin{align*}
        V^{1}_{k}:=\begin{cases}
        V^{0}_{k} ||B&\text{if }k=i\\
        V^{0}_{k} ||A&\text{if }k=j\\
        V^{0}_{k} ||C&\text{otherwise}
    \end{cases}.
    \end{align*}
    \\
    \\
    The induced subgraph of $V^{1}_1,\dots,V^{1}_{r}$ will not have the arrow $(i,j)$ anymore because the last segment violates the $SL_2$ rule. Otherwise, the graph will have the same arrows as the induced subgraph of $V^{0}_{1},\dots,V^{0}_{r}$.
    We can inductively remove all other unwanted arrows with the same strategy.\\
    The $V^{0}_{k}$ are all elements of $\Gamma_{2,4r-1}(\mathbb{Z})$. Each time we remove an arrow, we increase the length of the tuples by $8$. If we have to remove $s$ arrows, then $\Gamma_{2,4r+8s-1}$ has an induced subgraph isomorphic to $G$.
    \end{proof}
\end{Proposition}

\begin{Remark}
    The construction in the proof of Proposition \ref{embed finite digraphs} is incredibly wasteful. For example, the directed graph with $r$ vertices and no arrows would be realized as an induced subgraph of $\Gamma_{2,4r+8r(r-1)-1}(\mathbb{Z})$ this way. Instead, by picking $r$ different vertices of the form $(\begin{NiceMatrix}{}
        0&1&a&-1&0&1&0
    \end{NiceMatrix})$ with $a\neq 0$, one could already realize this as an induced subgraph of $\Gamma_{2,3}(\mathbb{Z})$.
\end{Remark}

\label{section_properties}
\newpage
\printbibliography
\end{document}